\documentclass[a4paper,11pt,reqno]{amsart}

\usepackage{url}
\usepackage[colorlinks=true, linkcolor=blue, citecolor=ForestGreen, urlcolor=blue]{hyperref}
\usepackage{cite}

\usepackage{latexsym,amsmath,amssymb,amsfonts,amsthm,stackengine}
\usepackage{mathrsfs}
\usepackage{mathtools}
\usepackage{dsfont}
\usepackage{bbm}
\usepackage{soul}
\usepackage{cancel}
\usepackage{esint}
\usepackage[utf8]{inputenc}

\usepackage[usenames, dvipsnames]{color}
\usepackage[usenames, dvipsnames, cmyk]{xcolor}
\usepackage{graphicx}
\usepackage{multirow}
\usepackage{enumerate}
\usepackage{enumitem}
\usepackage{subcaption}

\usepackage{a4wide}
\usepackage{epsfig,epstopdf}
\allowdisplaybreaks
\numberwithin{equation}{section}
\definecolor{grey}{rgb}{.7,.7,.7}
\definecolor{refkey}{gray}{.45}
\definecolor{labelkey}{gray}{.45}

\newtheorem{theorem}{Theorem}[section]
\newtheorem{proposition}[theorem]{Proposition}
\newtheorem{lemma}[theorem]{Lemma}

\newtheorem{corollary}[theorem]{Corollary}

\theoremstyle{remark}
\newtheorem{remark}[theorem]{Remark}
\theoremstyle{definition}

\newtheorem{example}[theorem]{Example}

\usepackage{tikz}
\usepackage{pgf,pgfplots}
\usetikzlibrary{calc}
\usetikzlibrary{decorations.markings}
\usetikzlibrary{decorations.pathmorphing}
\usetikzlibrary{decorations.shapes}
\usetikzlibrary{shapes,arrows,shapes.geometric,patterns,fadings}
\usetikzlibrary{arrows.meta,patterns.meta}
\usetikzlibrary{pgfplots.fillbetween}

\definecolor{bubbleblue}{RGB}{217,217,255}
\definecolor{bubbleorange}{RGB}{255,191,128}
\definecolor{labelgreen}{RGB}{0,89,0}
\definecolor{gridgreen}{RGB}{102,155,102}
\definecolor{hatchgray}{RGB}{209,209,209}

\newcommand{\e}{\varepsilon}
\newcommand{\vphi}{\varphi}

\newcommand{\N}{\mathbb N}

\newcommand{\R}{\mathbb R}

\newcommand{\dd}{\,\mathrm{d}}

\renewcommand{\setminus}{\backslash}

\newcommand{\defeq}{\coloneqq}

\newcommand{\E}{\mathcal{E}_q}

\newcommand{\en}{\mathcal{E}_q}
\newcommand{\nl}{\mathcal{D}}

\newcommand{\abs}[1]{\left\vert{#1}\right\vert}

\newcommand{\ba}{\begin{array}}
\newcommand{\ea}{\end{array}}

\newcommand{\tld}[1]{\widetilde{#1}}
\newcommand{\bthm}{\begin{theorem}}
\newcommand{\ethm}{\end{theorem}}
\newcommand{\bprop}{\begin{proposition}}
\newcommand{\eprop}{\end{proposition}}
\newcommand{\blemma}{\begin{lemma}}
\newcommand{\elemma}{\end{lemma}}
\newcommand{\bexmpl}{\begin{example}}
\newcommand{\eexmpl}{\end{example}}
\newcommand{\norm}[1]{\left\|#1\right\|}

\newcommand{\beqn}{\begin{equation}}
\newcommand{\eeqn}{\end{equation}}
\newcommand{\beqns}{\begin{equation*}}
\newcommand{\eeqns}{\end{equation*}}

\newcommand{\pt}{\partial}

\newcommand{\iarrow}{\hookrightarrow}

\newcommand{\ol}{\overline}

\newcommand{\Htwo}{\mathcal{H}^2}

\renewcommand{\leq}{\leqslant}
\renewcommand{\geq}{\geqslant}

\definecolor{mygreen}{rgb}{0.1,0.75,0.2}

\newcommand{\Om}{\Omega}

\newcommand{\etalchar}[1]{$^{#1}$}

\newcounter{myenumi}
\DeclareMathOperator{\id}{Id}

\DeclareMathOperator*{\dist}{dist}
\DeclareMathOperator{\loc}{loc}

\DeclareMathOperator{\capa}{cap}
\title[Minimality and stability for the charge-qubit energy]{Minimality and stability of the ball for the charge-qubit energy}

 \author{Ihsan Topaloglu}
 \address{Department of Mathematics and Statistics, Virginia Commonwealth University, Richmond, VA, United States}
 \email{iatopaloglu@vcu.edu}

\date{\today}


\begin{document}

\begin{abstract}
We study the shape optimization problem introduced by Mazzoleni, Muratov, and Ruffini as a model for a superconducting charge qubit, in which the first Dirichlet eigenvalue of a set of prescribed volume competes with the Coulomb self-energy of the ground state. We prove that, for sufficiently small charge, the ball is the unique minimizer, and that the energy deficit of a volume-preserving and centered perturbation of the ball is comparable to the squared $H^{1/2}$ norm of the boundary graph, with constants independent of the charge. We also obtain the corresponding results for the quadratic Hartree--Ohta--Kawasaki model.
\end{abstract}

\maketitle

%%%%%%%%%%%%%%%%%%%%%%%%%%%%%%%%%%%%%%%%%%%%%%%%%%%%%%%%%%%%%%%%%%%%%%%%%%%%%%
%%%%%%%%%%%%%%%%%%%%%%%%%%%%%%%%%%%%%%%%%%%%%%%%%%%%%%%%%%%%%%%%%%%%%%%%%%%%%%
\section{Introduction}\label{sec:intro}
%%%%%%%%%%%%%%%%%%%%%%%%%%%%%%%%%%%%%%%%%%%%%%%%%%%%%%%%%%%%%%%%%%%%%%%%%%%%%%
%%%%%%%%%%%%%%%%%%%%%%%%%%%%%%%%%%%%%%%%%%%%%%%%%%%%%%%%%%%%%%%%%%%%%%%%%%%%%%

For the functional
	\[
		\E(u,\Om) \defeq \int_\Om |\nabla u|^2 \dd x + \frac{q}{2}\int_\Om \int_\Om \frac{u^2(x) u^2(y)}{|x-y|}\dd x \dd y,
	\]
where $q\geq0$ is a fixed parameter, referred to as the charge, $\Om\subset\R^3$ is a measurable set of finite measure (more precisely, a quasi-open set; see below), and $u\in H_0^1(\Om)$, we define the energy
	\[
		\en(\Om) \defeq \inf \Big\{ \E(u,\Om) \colon u\in H^1_0(\Om),\ \int_\Om u^2 \dd x = 1\Big\}.
	\]
We are interested in the shape optimization problem
	\begin{equation}\label{eq:min-min_prob}
		e_q(m) \defeq \inf \big\{ \en(\Om) \colon |\Om|=m \big\}.
	\end{equation}
By scaling of the energy (see Remark~\ref{rem:scaling}), without loss of generality, we take $m$ to be the volume of the unit ball throughout this note. That is, we consider the problem $e_q\big(|B_1|\big)$, where $B_1\subset\R^3$ denotes the ball of radius one centered at the origin. Our main result is the characterization of the global minimizers as balls up to translations and sets of Sobolev capacity zero.

\begin{theorem}\label{thm:global}
There exists $q_0>0$ such that for every $0\leq q\leq q_0$ we have $e_q(|B_1|)=\en(B_1)$, and the minimizers of \eqref{eq:min-min_prob} with $m=|B_1|$ are exactly the translates of $B_1$, up to sets of Sobolev capacity zero.
\end{theorem}

For $q=0$ the energy $\en(\Om)$ is the first Dirichlet eigenvalue of $\Om$, and its minimality on balls among sets of prescribed volume is the classical Faber--Krahn inequality. For $q>0$ the Dirichlet term, which favors round sets, competes with the Coulomb self-interaction of the ground state, which favors spreading. Hence, the problem \eqref{eq:min-min_prob} can be considered as another example of geometric variational problems with competing short-range attractive and long-range repulsive interactions, which have received significant attention in the last two decades. 

The problem \eqref{eq:min-min_prob} was introduced by Mazzoleni, Muratov and Ruffini in \cite{MMR} as an optimal design problem for a superconducting charge qubit. Such a device, also called a Cooper-pair box, consists of a nanoscale superconducting island $\Omega$ embedded in an insulator, and quantum information is stored in the number of Cooper pairs on the island \cite{Bouchiat98,Kjaergaard20}. In the standard description the island is a capacitor and its shape enters only through the capacitance. When the island is small, the kinetic energy of the pairs competes with their mutual Coulomb repulsion. In \cite{MMR} the authors model the ground state of the pairs by the Hartree functional on $\Omega$ with Dirichlet boundary conditions, where the dimensionless parameter $q$ absorbs the number of pairs and all physical constants. The optimal design problem then asks which island of a given volume has the lowest ground-state energy.
%In \cite{MMR} this is interpreted as the shape that supports the largest number of stable charge states at a fixed gate voltage.

The main result of \cite{MMR} is an existence and regularity theorem in the small-charge regime. Indeed, the authors prove that for every $\e>0$ there exist $q^*=q^*(\e)>0$ and $\alpha=\alpha(\e)\in(0,1)$ satisfying the following: If $0<q\leq q^*$, then there exists an optimal set for the problem \eqref{eq:min-min_prob} with $m=|B_1|$, and every optimal set $\Om$ is $C^{2,\alpha}$-nearly spherical, namely, up to a translation,
	\begin{equation}\label{eq:MMR}
		\pt \Om = \big\{(1+\vphi(x))x \colon x\in\pt B_1\big\}, \qquad \vphi\in C^{2,\alpha}(\pt B_1), \qquad \norm{\vphi}_{C^{2,\alpha}(\pt B_1)} \leq \e,
	\end{equation}
see \cite[Theorem 1.2]{MMR}. In \cite[Remark 1.3]{MMR} the authors conjecture that for small $q$ the ball is the unique minimizer. Furthermore they show that in the large-charge regime balls are not minimizers \cite[Theorem 1.5]{MMR}. We also refer to \cite{MMR26N} for nonexistence results for screened Hartree energies.

Theorem~\ref{thm:global} settles the conjecture stated in \cite[Remark 1.3]{MMR}. The same conclusion is obtained in the anonymous preprint \cite{Anon} through the boundary flux of the ground state and the implicit function theorem. An argument of that type classifies all nearby volume-constrained critical domains and not only the minimizers. On the other hand, we prove Theorem~\ref{thm:global} by a direct comparison argument which requires no shape derivative of the optimized state, and which yields at the same time, without any second-order expansion of the minimized energy, a two-sided quantitative estimate for the energy deficit of nearly spherical sets (see Theorem~\ref{thm:local} below).

Throughout we adopt the quasi-open formulation of \cite[Section 2.3]{MMR}: we define $H^1_0(\Om)$ for quasi-open sets $\Om$ and identify admissible sets up to Sobolev capacity zero. When the boundary of a minimizer is discussed, its regular representative is understood (cf. Section~\ref{sec:global}). For functions extended by zero outside their domain we set
	\[
		\nl(f,g) \defeq \frac12 \int_{\R^3} \int_{\R^3} \frac{f(x)g(y)}{|x-y|}\dd x \dd y,
	\]
so that
	\begin{equation}\label{eq:energy}
		\E(u,\Om) = \int_\Om |\nabla u|^2 \dd x + q\nl(u^2,u^2).
	\end{equation}

Our main global minimality result relies on a stability result, and to state this stability result we need the following notation. Let $u_q$ be the positive normalized minimizer in the definition of $\en(B_1)$ (see Lemma~\ref{lem:ball} for its existence, uniqueness, and radial symmetry). Fix $\beta\in(0,1)$ and let 
	\[
		X\defeq C^{1,\beta}(\pt B_1).
	\]
We parametrize nearly spherical sets as radial graphs over the unit sphere,
	\[
		\Om_\vphi \defeq \big\{r\theta \colon 0\leq r<1+\vphi(\theta),\ \theta\in\pt B_1\big\}, \qquad \vphi\in X,
	\]
and by $T_\vphi\colon\R^3\to\R^3$ we denote the bi-Lipschitz diffeomorphism with $T_\vphi(B_1)=\Om_\vphi$ constructed in Lemma~\ref{lem:extension}. This map extends the boundary map $\theta\mapsto(1+\vphi(\theta))\theta$ linearly in $\vphi$, and its distance from the identity is controlled in $C^1$ by $\norm{\vphi}_X$ and in $H^1$ by $\norm{\vphi}_{H^{1/2}}$.

\begin{theorem}\label{thm:local}
There exist $q_0,\delta,\kappa,C>0$ with the following property. If
	\begin{equation}\label{eq:constraints}
		0\leq q\leq q_0, \qquad \norm{\vphi}_X\leq\delta, \qquad |\Om_\vphi|=|B_1|, \qquad \int_{\Om_\vphi}x\dd x=0,
	\end{equation}
then every $v\in H^1_0(\Om_\vphi)$ with $\norm{v}_2=1$ satisfies
	\begin{multline}\label{eq:joint}
		\E(v,\Om_\vphi)-\en(B_1) \\ \geq \kappa\Big(\norm{\vphi}_{H^{1/2}}^2 + \min_{\sigma\in\{-1,1\}}\norm{v\circ T_\vphi-\sigma u_q}_{H^1(B_1)}^2\Big) + q\nl(v^2-u_q^2,v^2-u_q^2).
	\end{multline}
Moreover,
	\begin{equation}\label{eq:two-sided}
		\kappa\norm{\vphi}_{H^{1/2}}^2 \leq \en(\Om_\vphi)-\en(B_1) \leq C\norm{\vphi}_{H^{1/2}}^2.
	\end{equation}
\end{theorem}

\bigskip

\begin{remark}\label{rem:L2} In \eqref{eq:joint} the transported $H^1$ distance may be replaced by $\norm{v-\sigma u_q}_2$. That is,
	\[
		\E(v,\Omega_\varphi)-\E(B_1) \geq \kappa\Big(\norm{\varphi}_{H^{1/2}}^2+\min_{\sigma\in\{-1,1\}}\norm{v-\sigma u_q}_2^2\Big)+q\,\nl(v^2-u_q^2,\,v^2-u_q^2),
	\]
where $v$ and $u_q$ are extended by zero. If $v\geq0$, the choice $\sigma=1$ is admissible in both versions of \eqref{eq:joint}.
\end{remark}

We note that Theorem~\ref{thm:local} contains no minimality or criticality assumption: it concerns every normalized state on every set in the neighborhood \eqref{eq:constraints}, and it uses neither \eqref{eq:MMR} nor any regularity of an optimal state on $\Om_\vphi$. The last term in \eqref{eq:joint} is nonnegative, the Coulomb form being positive definite, and we retain it because the identity from which \eqref{eq:joint} is derived is exact. As in every local argument of this type, the quantitative estimates are restricted to the neighborhood \eqref{eq:constraints}. The near-sphericity of actual minimizers in \eqref{eq:MMR} does not reduce arbitrary competitors to this neighborhood, and we do not claim a deficit inequality without a proximity assumption. The constants in \eqref{eq:joint} and \eqref{eq:two-sided} are independent of the charge, and at $q=0$ the lower bound in \eqref{eq:two-sided} reduces to the quantitative Faber--Krahn inequality for nearly spherical sets, which is the local step of the selection principle in \cite{BDV} (see also \cite{AKN23}).

The proof of Theorem~\ref{thm:local} follows by the combination of three main ingredients: (a) an exact comparison identity, obtained by completing the Coulomb square around $u_q$ and by continuing $u_q$ radially across the sphere as a solution of its own linear equation, so that the comparison function is defined on every nearby set (Lemma~\ref{lem:comparison}); (b) a coercivity estimate on the fixed ball for the Robin form which arises from the boundary term of this identity, in which the degenerate directions, namely the radial one and the three translations, are not projected away but retained with explicit defects (Lemma~\ref{lem:robin}); (c) the observation that, once the state is decomposed orthogonally along the continued ground state, the volume and barycenter constraints make all these defects small multiples of the remainder, so that no additive error depending only on $q$ appears (Section~\ref{sec:proof}). The third point is where the argument departs from the decomposition that one would first try, which we now describe.

\medskip

\begin{remark}[The decomposition about the Dirichlet ground state]\label{rem:dirichlet}
The strategy suggested by the quantitative Faber--Krahn inequality is the following. Let $\Om$ be a minimizer, let $u_{q,\Om}$ be a nonnegative optimal state on $\Om$, and let $\lambda_1(\Om)$ and $\psi_\Om$ denote the first Dirichlet eigenvalue and the positive normalized first eigenfunction of $\Om$. Decompose $u_{q,\Om}=a\psi_\Om+r$ with $r\perp\psi_\Om$ in $L^2(\Om)$. Comparing $u_{q,\Om}$ with $\psi_\Om$ in the minimization defining $\en(\Om)$, and using that the spectral gap of $\Om$ is bounded below uniformly by \eqref{eq:MMR}, gives $\norm{r}_{H^1(\Om)}\leq Cq$, hence
	\[
		\en(\Om) = \lambda_1(\Om) + q\nl(\psi_\Om^2,\psi_\Om^2) + O(q^2).
	\]
Since $\en(\Om)\leq\en(B_1)\leq\pi^2+q\nl(\psi^2,\psi^2)$, the sharp quantitative Faber--Krahn inequality of \cite{BDV,AKN23} applied to $\lambda_1(\Om)-\pi^2$, together with the Lipschitz dependence of the Coulomb energy on the state and of $\psi_\Om$ on the graph of $\Om$, yields
	\[
		|\Om\triangle B_1| + \norm{u_{q,\Om}-\psi}_2 \leq Cq
	\]
after a translation. As stated, this does not imply rigidity. Moreover, no refinement of the estimates can turn it into a rigidity estimate, because the error $O(q^2)$ in the expansion of $\en(\Om)$ does not vanish when $\Om=B_1$, and is therefore unrelated to any defect of $\Om$. Indeed, let $w\in H^1_0(B_1)$ be the solution of
	\[
		(-\Delta-\pi^2)w = -\Big(K(\psi^2)-\int_{B_1}K(\psi^2)\,\psi^2\dd x\Big)\psi \quad\text{in } B_1,
	\]
which is orthogonal to $\psi$, and set
	\[
		N \defeq \int_{B_1}|\nabla w|^2\dd x - \pi^2\int_{B_1}w^2\dd x.
	\]
Then $N>0$, since $w\perp\psi$ and $w\neq0$, the potential $K(\psi^2)$ being nonconstant in $B_1$. Second-order perturbation theory gives $u_q=\psi+qw+O(q^2)$ in $H^1(B_1)$,
	\[
		\int_{B_1}|\nabla u_q|^2\dd x = \pi^2 + q^2N + O(q^3), \qquad
		\nl(u_q^2,u_q^2) = \nl(\psi^2,\psi^2) - 2qN + O(q^2),
	\]
and consequently
	\[
		\en(B_1) = \pi^2 + q\nl(\psi^2,\psi^2) - q^2N + O(q^3).
	\]
That is, the optimized state lowers the Coulomb energy at first order in $q$ at the cost of raising the kinetic energy at second order only, and the decomposition about $\psi_\Om$ records this gain of order $q^2$ as an error, for every $\Om$ and in particular for the ball. The comparison identity of Lemma~\ref{lem:comparison} expands about $u_q$ itself, and this adjustment is then accounted for exactly.
\end{remark}

The paper is organized as follows. Section~\ref{sec:ball} contains the properties of the ball state, its radial continuation, and the exact comparison identity. Section~\ref{sec:fixed} contains the two fixed-domain ingredients, the boundary extension and the Robin coercivity estimate. Theorem~\ref{thm:local} is proved in Section~\ref{sec:proof} and Theorem~\ref{thm:global} in Section~\ref{sec:global}. The quadratic model of \cite{MMR26H} is treated in Section~\ref{sec:quadratic}.

\subsection*{Notation}
We denote by
	\begin{equation}\label{eq:first-eigenfunction}
			\psi(x) \defeq \frac{\sin(\pi|x|)}{\sqrt{2\pi}\,|x|}
	\end{equation}
the positive normalized first Dirichlet eigenfunction of $B_1$, with eigenvalue $\pi^2$. The value at the origin is defined by continuity. Points of $\pt B_1$ are denoted by $\theta$, and integrals over $\pt B_1$ are written with respect to $\Htwo(\theta)$. For a function extended by zero outside its domain, $\norm{\cdot}_p$ denotes the $L^p(\R^3)$ norm. For a function on $\pt B_1$ all norms, including $\norm{\cdot}_{L^2}$ and $\norm{\cdot}_{H^{1/2}}$, are taken on $\pt B_1$. For a function $f$ extended by zero outside its domain we denote its Coulomb potential by
	\[
		Kf(x) \defeq \int_{\R^3}\frac{f(y)}{|x-y|}\dd y,
	\]
so that
	\begin{equation}\label{eq:coulomb-identities}
		2\nl(f,g) = \int_{\R^3} g\,Kf\dd x, \qquad \nl(f,f) = \frac{1}{8\pi}\int_{\R^3}\abs{\nabla Kf}^2\dd x \geq 0
	\end{equation}
for all real $f$, $g$ of finite Coulomb energy (see \cite[Theorem 9.8]{LiebLoss}). Constants denoted by $C$, $c$, $\kappa$ are positive and independent of $q$, $\vphi$, and the states involved. Their value may change from line to line.

\medskip

\subsection*{Note on the LLM usage} AI tools were used in this project. In particular, the boundary extension lemma (Lemma~\ref{lem:extension}) and the Robin coercivity (Lemma~\ref{lem:robin}) were obtained through various interactions with LLMs with rather minimal input of ideas. These are the two technical ingredients needed to turn the identity of Lemma~\ref{lem:comparison} into a lower bound. Discussions with LLMs were also essential in the formulation of Section~\ref{sec:quadratic}. All mathematical arguments and proofs in the final manuscript were checked and written by the author.

\medskip

\subsection*{Note on a related paper} The preprint \cite{MMR26R} by Mazzoleni, Moraschi, and Ruffini was posted on arXiv two days after the first version of this paper. In their paper the authors prove that for sufficiently small $q$ the ball is the unique minimizer of the quadratic Hartree--Ohta--Kawasaki energy introduced in \cite{MMR26H}, settling their earlier conjecture. The two works are
independent, and we thank the authors for bringing theirs to our
attention. Their proof is a second shape derivative argument, with an
elegant continuity step at its core. The second variation of the energy
along normal graphs of the ball is shown to converge, in operator norm
on $H^{1/2}(\partial B_1)$, to that of the first Dirichlet eigenvalue as
$q\to0$, which gives precisely the uniform control of the Taylor
remainder that pointwise convergence cannot. This leads to a clean
$H^{1/2}$ lower bound for a suitably penalized energy, and, as the
authors point out, the strategy is adaptable to a broad class of
nonlocal shape functionals. The same rigidity statement for the
quadratic model is obtained in Section~\ref{sec:quadratic} below by the
comparison argument of this paper, which does not use shape derivatives.

\bigskip
%%%%%%%%%%%%%%%%%%%%%%%%%%%%%%%%%%%%%%%%%%%%%%%%%%%%%%%%%%%%%%%%%%%%%%%%%%%%%%
%%%%%%%%%%%%%%%%%%%%%%%%%%%%%%%%%%%%%%%%%%%%%%%%%%%%%%%%%%%%%%%%%%%%%%%%%%%%%%
\section{The ball state and an exact comparison}\label{sec:ball}
%%%%%%%%%%%%%%%%%%%%%%%%%%%%%%%%%%%%%%%%%%%%%%%%%%%%%%%%%%%%%%%%%%%%%%%%%%%%%%
%%%%%%%%%%%%%%%%%%%%%%%%%%%%%%%%%%%%%%%%%%%%%%%%%%%%%%%%%%%%%%%%%%%%%%%%%%%%%%

We define the potential
	\[
		V_q \defeq K(u_q^2), \qquad\text{that is,}\qquad V_q(x) = \int_{B_1}\frac{u_q^2(y)}{|x-y|}\dd y,
	\]
and the Lagrange multiplier
	\[
		\mu_q \defeq \int_{B_1}|\nabla u_q|^2\dd x + q\int_{B_1}V_qu_q^2\dd x = \en(B_1) + q\nl(u_q^2,u_q^2)>\en(B_1)
	\]
for $q>0$ so that
	\begin{equation}\label{eq:ball-state}
		-\Delta u_q + qV_qu_q = \mu_qu_q \quad\text{in } B_1, \qquad u_q=0 \quad\text{on } \pt B_1.
	\end{equation}
The lemma below shows that the ball state can be continued across the sphere as a solution of the frozen equation, with bounds and a boundary expansion that are uniform in the charge.

\begin{lemma}\label{lem:ball}
For $q\geq0$ sufficiently small the following hold.
\begin{enumerate}
\item The ball has a unique nonnegative normalized minimizing state $u_q$. It is positive in $B_1$, radial, and a first eigenfunction of the Dirichlet operator $-\Delta+qV_q$, with eigenvalue $\mu_q$.
\item The estimates
	\begin{equation}\label{eq:ball-estimates}
		\norm{u_q-\psi}_{C^2(\ol{B}_1)} + |\mu_q-\pi^2| \leq Cq, \qquad \norm{V_q}_{C^1(\R^3)}\leq C
	\end{equation}
	hold uniformly in $q$.
\item There is a radial function $\Psi_q\in C^2(\ol{B}_2)$, with $\Psi_q=u_q$ on $B_1$, which solves
	\begin{equation}\label{eq:continuation}
		-\Delta\Psi_q + (qV_q-\mu_q)\Psi_q = 0 \quad\text{in } B_2.
	\end{equation}
It is bounded in $C^2(\ol{B}_2)$ and in $C^3(\ol{B}_2\setminus B_{1/2})$ uniformly in $q$. Moreover, for $c_q\defeq-\Psi_q'(1)$, we have
	\begin{equation}\label{eq:boundary-taylor}
		c_q = \sqrt{\pi/2}+O(q), \qquad \Psi_q(1+s) = -c_qs + c_qs^2 + O(|s|^3) \quad\text{uniformly for } |s|\leq1.
	\end{equation}
\end{enumerate}
\end{lemma}

\begin{proof}
(i) The direct method gives a nonnegative normalized minimizer $u$ of $\E(\cdot,B_1)$, as in \cite[Lemma 2.2]{MMR}. Fix such a $u$ and set
	\[
		V_u \defeq K(u^2), \qquad
		\mu \defeq \int_{B_1}|\nabla u|^2\dd x + q\int_{B_1}V_u u^2\dd x.
	\]
Recall that $\norm{K(f^2)}_\infty\leq2\norm{f}_2\norm{\nabla f}_2$ for $f\in H^1(\R^3)$, by Hardy's inequality (cf. \cite[Lemma 2.1]{MMR}). Hence, $V_u$ is bounded, and the first variation of $\E(\cdot,B_1)$ under the $L^2$-constraint gives
	\[
		-\Delta u + qV_uu = \mu u \quad\text{in } B_1, \qquad u=0 \quad\text{on } \pt B_1.
	\]
Since $u$ is smooth in $B_1$ (cf.\ \cite[Lemma 2.4]{MMR}) and nonnegative, the strong maximum principle (see, for example, \cite[Theorem 3.5]{GT}) yields that $u>0$ in $B_1$. A positive Dirichlet eigenfunction of $-\Delta+qV_u$ is a first eigenfunction, so $\mu$ is the first eigenvalue of this operator, which is simple.

We claim that every normalized minimizer equals either $u$ or $-u$. Let $v\in H^1_0(B_1)$ with $\norm{v}_2=1$. Expanding the Coulomb term about $u^2$,
	\[
		\nl(v^2,v^2) = \nl(u^2,u^2) + \int_{B_1}V_u\,(v^2-u^2)\dd x + \nl(v^2-u^2,v^2-u^2),
	\]
and since $\int_{B_1}V_uu^2\dd x=\mu-\int_{B_1}|\nabla u|^2\dd x$, this gives the exact identity
	\begin{equation}\label{eq:square}
		\E(v,B_1)-\E(u,B_1) = \Big(\int_{B_1}\big(|\nabla v|^2+qV_uv^2\big)\dd x - \mu\Big) + q\nl(v^2-u^2,v^2-u^2).
	\end{equation}
The first term is nonnegative by the Rayleigh principle for $-\Delta+qV_u$, and vanishes only if $v$ is a first eigenfunction, that is, $v=\pm u$ by simplicity and normalization. The second term is nonnegative by \eqref{eq:coulomb-identities}. Hence, if $v$ is also a minimizer, both terms in \eqref{eq:square} vanish and $v=\pm u$. This proves the claim, and in particular there is exactly one nonnegative normalized minimizer, which we denote by $u_q$. Then $V_u=V_q$ and $\mu=\mu_q$.

Finally, every rotation of $u_q$ is again a nonnegative normalized minimizer, since $\E(\cdot,B_1)$ is rotation invariant, so $u_q$ is radial (cf. \cite[Proposition 2.9 and Corollary 2.10]{MMR}).

\medskip

(ii) Since $\psi\in H^1_0(B_1)$ is normalized, it is an admissible competitor in the minimization defining $\en(B_1)$, and since the Coulomb term is nonnegative and $\pi^2$ is the first Dirichlet eigenvalue of $B_1$, we have
	\begin{equation}\label{eq:ball-energy-bounds}
		\pi^2 \leq \int_{B_1}|\nabla u_q|^2\dd x \leq \E(u_q,B_1) = \en(B_1) \leq \E(\psi,B_1) = \pi^2 + q\nl(\psi^2,\psi^2) \leq \pi^2 + Cq.
	\end{equation}
Moreover $\nl(u_q^2,u_q^2)\leq C$, by the bound on $K(f^2)$ recalled above; hence $|\mu_q-\pi^2|\leq Cq$.

Now write $u_q=a\psi+\eta$ with $\eta\perp\psi$ in $L^2(B_1)$, so that $a\geq0$ and $a^2+\norm{\eta}_2^2=1$. Since $\eta$ is orthogonal to $\psi$ in the Dirichlet form as well, and since the second Dirichlet eigenvalue of $B_1$ exceeds $\pi^2$, we have
	\begin{equation}\label{eq:kinetic-excess}
		\int_{B_1}|\nabla u_q|^2\dd x - \pi^2 = \int_{B_1}|\nabla \eta|^2\dd x - \pi^2\norm{\eta}_2^2 \geq c\norm{\eta}_{H^1(B_1)}^2.
	\end{equation}
By the Lipschitzianity of the Coulomb energy, for $v,u$ in a bounded subset of $H^1(\R^3)$, we get
	\begin{equation}\label{eq:D-lip}
		\abs{\nl(v^2,v^2)-\nl(u^2,u^2)} = \abs{\nl(v^2-u^2,v^2+u^2)} \leq C\norm{v^2-u^2}_1 \leq C\norm{v-u}_2.
	\end{equation}
Since \eqref{eq:ball-energy-bounds} implies
	\[
		\int_{B_1}|\nabla u_q|^2\dd x - \pi^2 \leq q\big(\nl(\psi^2,\psi^2)-\nl(u_q^2,u_q^2)\big),
	\]
combining this with \eqref{eq:kinetic-excess}, \eqref{eq:D-lip}, and the bound $\norm{u_q-\psi}_2\leq(1-a)+\norm{\eta}_2\leq2\norm{\eta}_2$, we obtain
	\[
		c\norm{\eta}_{H^1(B_1)}^2 \leq Cq\norm{\eta}_2.
	\]
Hence $\norm{\eta}_{H^1(B_1)}\leq Cq$, $1-a\leq Cq^2$, and $\norm{u_q-\psi}_{H^1(B_1)}\leq Cq$.

The $C^2$ bound follows from elliptic regularity. By \eqref{eq:ball-state}, we have
	\begin{equation}\label{eq:ball-difference}
		(-\Delta-\pi^2)\eta = (-\Delta-\pi^2)(u_q-\psi) = (\mu_q-\pi^2)u_q - qV_qu_q =: f_q.
	\end{equation}
Since $|\mu_q-\pi^2|\leq Cq$ and $V_q$ is bounded, $\norm{f_q}_{L^2(B_1)}\leq Cq$. Hence, the global $H^2$ estimate \cite[Theorem 8.12]{GT} implies that $\norm{\eta}_{H^2(B_1)}\leq C(\norm{\eta}_2+\norm{f_q}_2)\leq Cq$. Since $H^2(B_1)\iarrow L^\infty(B_1)$ in three dimensions, this gives $\norm{\eta}_\infty\leq Cq$.  Hence, 
	\[
		\norm{u_q}_\infty\leq a\norm{\psi}_\infty+\norm{\eta}_\infty\leq C
	\]
uniformly in $q$. Consequently $f_q$ and $\pi^2\eta$ are bounded in $L^\infty(B_1)$ by $Cq$, and so we have that $\norm{f_q+\pi^2\eta}_{L^p(B_1)}\leq Cq$ for every $p<\infty$.

 Now, since $\eta\in H^1_0(B_1)$ solves $\Delta \eta=-f_q-\pi^2\eta$, by the $W^{2,p}$-theory for the Dirichlet problem (see \cite[Theorem 9.15 and Lemma 9.17]{GT}), we obtain 
 	\[
 		\norm{\eta}_{W^{2,p}(B_1)}\leq C\norm{f_q+\pi^2\eta}_{L^p(B_1)}\leq Cq.
 	\]
In particular, taking $p>3$ and using $W^{2,p}(B_1)\hookrightarrow C^1(\overline B_1)$, we get that $\nabla u_q$ is bounded uniformly in $q$. Since $u_q^2$ is bounded with support in $\ol{B}_1$, by \cite[Lemma 4.1]{GT}, $V_q\in C^1(\R^3)$ with $\norm{V_q}_{C^1}\leq C\norm{u_q}_\infty^2$. Hence, $\|f_q\|_{C^{0,\beta}(\ol{B}_1)}\leq C(|\mu_q-\pi^2|+q)\leq Cq$. We also have that $\|\pi^2\eta\|_{C^{0,\beta}(\ol{B}_1)}\leq C\|\eta\|_{W^{2,p}(B_1)}\leq Cq$. Now the Schauder theory for the Dirichlet problem on the ball (see \cite[Theorems 6.6 and 6.14]{GT}), applied once more to $\Delta \eta=-f_q-\pi^2\eta$, yields $\norm{\eta}_{C^{2,\beta}(\ol{B}_1)}\leq Cq$. Together with $1-a\leq Cq^2$ this proves $\norm{u_q-\psi}_{C^{2,\beta}(\ol{B}_1)}\leq Cq$.

\medskip

(iii) Since $u_q$ and $V_q$ are radial, in the radial variable $\rho=|x|$, equation \eqref{eq:ball-state} becomes
	\begin{equation}\label{eq:radial-ode}
		-\Psi'' - \frac{2}{\rho}\Psi' + \big(qV_q(\rho)-\mu_q\big)\Psi = 0,
	\end{equation}
and $u_q$ solves \eqref{eq:radial-ode} on $(0,1]$ with $u_q(1)=0$. Let $\Psi_q$ be the solution of \eqref{eq:radial-ode} on $(0,2]$ with $\Psi_q=u_q$ on $(0,1]$. On $[1,2]$ it is the solution of the Cauchy problem with data $\Psi_q(1)=0$ and $\Psi_q'(1)=u_q'(1)$, which exists on the whole interval since the equation is linear. We regard $\Psi_q$ as a radial function on $B_2$.

We now verify the bounds. On $\ol{B}_1$ we have $\Psi_q=u_q$, which is bounded in $C^2(\ol{B}_1)$ by \eqref{eq:ball-estimates}. On $[1/2,2]$ the coefficient $qV_q-\mu_q$ of \eqref{eq:radial-ode} is bounded in $C^1$ by \eqref{eq:ball-estimates}, and the Cauchy data at $\rho=1$ are bounded, so $\Psi_q$ is bounded in $C^3([1/2,2])$ uniformly in $q$ by continuous dependence. These match to second order at $\rho=1$ because both solve \eqref{eq:radial-ode}. Hence $\Psi_q\in C^2(\ol{B}_2)$, with the uniform bounds in $C^2(\ol{B}_2)$ and $C^3(\ol{B}_2\setminus B_{1/2})$ stated in (iii), and $\Psi_q$ solves \eqref{eq:continuation}.

Finally we compute the expansion at the sphere. Since $\Psi_q(1)=0$, evaluating \eqref{eq:radial-ode} at $\rho=1$ gives $\Psi_q''(1)=-2\Psi_q'(1)=2c_q$, so Taylor's formula with the uniform $C^3$ bound on $[1/2,3/2]$ yields
	\[
		\Psi_q(1+s) = -c_qs + c_qs^2 + O(|s|^3) \qquad\text{for } |s|\leq1/2,
	\]
and for $1/2\leq|s|\leq1$ the same expansion holds trivially, as both sides are $O(1)$. Together with $c_q=-u_q'(1)=-\psi'(1)+O(q)=\sqrt{\pi/2}+O(q)$, which follows from \eqref{eq:ball-estimates} and $\psi'(1)=-\sqrt{\pi/2}$, this proves \eqref{eq:boundary-taylor}.
\end{proof}

The next lemma is the main tool in the proof of Theorem~\ref{thm:local}. It gives an exact expression for the energy deficit with respect to the ball of an arbitrary normalized state on an arbitrary nearby set.

\begin{lemma}\label{lem:comparison}
Let $\Om\subset B_2$ be a bounded Lipschitz domain and let $v\in H^1_0(\Om)$ with $\norm{v}_2=1$. For every $\alpha\in\R$ set $r\defeq v-\alpha\Psi_q$ in $\Om$. Then
	\begin{multline}\label{eq:comparison}
		\E(v,\Om)-\en(B_1) = \int_\Om\big(|\nabla r|^2+(qV_q-\mu_q)r^2\big)\dd x \\ - \alpha^2\int_{\pt\Om}\Psi_q\,\pt_\nu\Psi_q\dd\Htwo(x) + q\nl(v^2-u_q^2,v^2-u_q^2).
	\end{multline}
\end{lemma}

\begin{proof}
Expanding the quadratic form $\nl$ about $u_q^2$, for the zero extensions of $v$ and $u_q$ we have
	\[
		\nl(v^2,v^2) = \nl(u_q^2,u_q^2) + \int_\Om V_qv^2\dd x - \int_{B_1}V_qu_q^2\dd x + \nl(v^2-u_q^2,v^2-u_q^2),
	\]
where we used $2\nl(u_q^2,f)=\int_{\R^3}V_qf\dd x$, by \eqref{eq:coulomb-identities}. By the definition of $\mu_q$, using $\norm{v}_2=1$, we have $q\int_{B_1}V_qu_q^2\dd x-q\nl(u_q^2,u_q^2)=\mu_q\int_\Om v^2\dd x -\en(B_1)$. Hence,
	\begin{equation}\label{eq:completed-square}
		\E(v,\Om)-\en(B_1) = \int_\Om\big(|\nabla v|^2+(qV_q-\mu_q)v^2\big)\dd x + q\nl(v^2-u_q^2,v^2-u_q^2).
	\end{equation}
Now we expand $v=\alpha\Psi_q+r$ in the quadratic form in \eqref{eq:completed-square}. Since $\Psi_q\in C^2(\ol{\Om})$ solves \eqref{eq:continuation}, integration by parts on the Lipschitz domain $\Om$ gives
	\begin{gather*}
		\int_\Om\big(|\nabla\Psi_q|^2+(qV_q-\mu_q)\Psi_q^2\big)\dd x = \int_{\pt\Om}\Psi_q\,\pt_\nu\Psi_q\dd\Htwo(x),\\
		\int_\Om\big(\nabla\Psi_q\cdot\nabla r+(qV_q-\mu_q)\Psi_qr\big)\dd x = \int_{\pt\Om}r\,\pt_\nu\Psi_q\dd\Htwo(x).
	\end{gather*}
On $\pt\Om$ the trace of $v$ vanishes, so that $r=-\alpha\Psi_q$ there, and the mixed term equals
	\[
		2\alpha\int_{\pt\Om}r\,\pt_\nu\Psi_q\dd\Htwo(x) = -2\alpha^2\int_{\pt\Om}\Psi_q\,\pt_\nu\Psi_q\dd\Htwo(x).
	\]
Adding these three contributions we obtain \eqref{eq:comparison}.
\end{proof}

\medskip

\begin{remark}
Here we note that $\Psi_q$ is not the zero extension of $u_q$. Indeed, by \eqref{eq:boundary-taylor} it changes sign just outside the ball. This, however, does not cause any difficulty since it is just a comparison function and not an admissible state, and no division by $\Psi_q$ is used anywhere.
\end{remark}

%%%%%%%%%%%%%%%%%%%%%%%%%%%%%%%%%%%%%%%%%%%%%%%%%%%%%%%%%%%%%%%%%%%%%%%%%%%%%%
%%%%%%%%%%%%%%%%%%%%%%%%%%%%%%%%%%%%%%%%%%%%%%%%%%%%%%%%%%%%%%%%%%%%%%%%%%%%%%
\section{The boundary extension and the Robin form}\label{sec:fixed}
%%%%%%%%%%%%%%%%%%%%%%%%%%%%%%%%%%%%%%%%%%%%%%%%%%%%%%%%%%%%%%%%%%%%%%%%%%%%%%
%%%%%%%%%%%%%%%%%%%%%%%%%%%%%%%%%%%%%%%%%%%%%%%%%%%%%%%%%%%%%%%%%%%%%%%%%%%%%%

The identity of Lemma~\ref{lem:comparison} is stated on the varying set $\Om$, whereas its lower bound will be obtained on the fixed ball $B_1$. In this section we collect the two ingredients needed for the passage from one to the other: a diffeomorphism from $B_1$ onto $\Om_\vphi$, under which the remainder in \eqref{eq:comparison} is transported, and a coercivity estimate on $B_1$ for the Robin form which the transported identity produces in the limit $q\to0$, $\vphi\to0$.

The first lemma extends the boundary map $\theta\mapsto(1+\vphi(\theta))\theta$ to a diffeomorphism of $\R^3$ in such a way that the extension is controlled by the strong norm $\norm{\vphi}_X$ in $C^1$, which keeps the change of variables uniformly bi-Lipschitz, and by the weak norm $\norm{\vphi}_{H^{1/2}}$ in $H^1$, which is what makes the transport error in Section~\ref{sec:proof} quadratic.

\begin{lemma}\label{lem:extension}
There exists a linear map $\vphi\mapsto T_\vphi-\id$, from functions on $\pt B_1$ to vector fields on $\R^3$ supported in $B_2$, such that $T_\vphi(\theta)=(1+\vphi(\theta))\theta$ on $\pt B_1$ and
	\begin{equation}\label{eq:map-bounds}
		\norm{T_\vphi-\id}_{C^{1,\beta}(\R^3)}\leq C\norm{\vphi}_X, \qquad \norm{T_\vphi-\id}_{H^1(\R^3)}\leq C\norm{\vphi}_{H^{1/2}}.
	\end{equation}
For $\norm{\vphi}_X$ sufficiently small, $T_\vphi$ is a $C^1$ diffeomorphism of $\R^3$, bi-Lipschitz with constants as close to one as necessary, and $T_\vphi(B_1)=\Om_\vphi$.
\end{lemma}

\begin{proof}
We first extend the boundary field $\theta\mapsto\vphi(\theta)\theta$ from $\pt B_1$ to $B_1$. Let $\xi_0$ be its componentwise harmonic extension. Expanding each component in spherical harmonics gives $\norm{\xi_0}_{H^1(B_1)}\leq C\norm{\vphi}_{H^{1/2}}$. For the H\"older bound, let $\eta$ be a smooth function on $[0,\infty)$ vanishing on $[0,1/4]$ and equal to one on $[1/2,\infty)$, and let $\Phi(x)\defeq\eta(|x|)\vphi(x/|x|)x$. Then $\Phi\in C^{1,\beta}(\ol{B}_1)$ coincides with $\vphi(\theta)\theta$ on $\pt B_1$ and satisfies $\norm{\Phi}_{C^{1,\beta}(\ol{B}_1)}\leq C\norm{\vphi}_X$. Since $\xi_0-\Phi\in H^1_0(B_1)$ and $\Delta\xi_0=0$, \cite[Theorems 8.33 and 8.34]{GT} give $\xi_0\in C^{1,\beta}(\ol{B}_1)$ with
	\[
		\norm{\xi_0}_{C^{1,\beta}(\ol{B}_1)} \leq C\big(\norm{\xi_0}_{L^\infty(B_1)}+\norm{\Phi}_{C^{1,\beta}(\ol{B}_1)}\big) \leq C\norm{\vphi}_X,
	\]
where $\norm{\xi_0}_{L^\infty(B_1)}\leq\norm{\vphi}_{L^\infty}$ by the maximum principle.

We next extend $\xi_0$ from $B_1$ to $\R^3$ across the sphere, by a reflection which preserves both bounds. Write $x=\rho\theta$ with $\rho=|x|$ and $\theta\in\pt B_1$, and let $\chi$ be a smooth function on $[0,\infty)$ equal to one on $[0,13/12]$ and vanishing on $[7/6,\infty)$. For each component $w$ of $\xi_0$ set
	\[
		\tilde w(\rho\theta) \defeq
		\begin{cases}
			w(\rho\theta) & \text{if } \rho\leq1,\\
			\chi(\rho)\big(6w((2-\rho)\theta) - 8w((3-2\rho)\theta) + 3w((4-3\rho)\theta)\big) & \text{if } \rho>1,
		\end{cases}
	\]
and let $\xi$ be the vector field with components $\tilde w$. The two expressions agree on $\pt B_1$ together with their first radial derivatives, so that $\xi\in C^{1,\beta}(\R^3)$, and $\xi\in H^1(\R^3)$ as well. Each of the three reflected terms is obtained from $w$ by a change of the radial variable with bounded Jacobian, hence
	\[
		\norm{\xi}_{C^{1,\beta}(\R^3)} \leq C\norm{\xi_0}_{C^{1,\beta}(\ol{B}_1)}, \qquad
		\norm{\xi}_{H^1(\R^3)} \leq C\norm{\xi_0}_{H^1(B_1)}.
	\]
By construction $\xi$ is supported in $B_{7/6}\subset B_2$, it is linear in $\vphi$, and $\xi(\theta)=\vphi(\theta)\theta$ on $\pt B_1$. Setting 
	\[
		T_\vphi\defeq\id+\xi,
	\]
the two bounds in \eqref{eq:map-bounds} follow from the estimates above.

Finally, if $\norm{\vphi}_X$ is small, then $\norm{D\xi}_\infty\leq1/2$, so $T_\vphi$ is a $C^1$ map with 
	\[
		\frac12|x-y|\leq|T_\vphi(x)-T_\vphi(y)|\leq\frac32|x-y|.
	\]
Since $\xi$ has compact support, $T_\vphi$ is proper, hence a bi-Lipschitz diffeomorphism of $\R^3$, with constants as close to one as desired. Furthermore, it maps $\pt B_1$ onto $\pt\Om_\vphi$ by the boundary identity, and since $\Om_\vphi$ is the bounded component of $\R^3\setminus\pt\Om_\vphi$, it maps $B_1$ onto $\Om_\vphi$.
\end{proof}

\bigskip

For $z\in H^1(B_1)$ we consider the Robin form
	\begin{equation}\label{eq:Q}
		Q(z) \defeq \int_{B_1}\big(|\nabla z|^2-\pi^2z^2\big)\dd x + 2\int_{\pt B_1}z^2\dd\Htwo(\theta).
	\end{equation}
It is the limit, as $q\to0$ and $\vphi\to0$, of the quadratic part of the comparison identity \eqref{eq:comparison} (cf. \eqref{eq:boundary-comparison} and \eqref{eq:bulk-comparison} below). Here the coefficient $2$ is the ratio $-\Psi_q''(1)/\Psi_q'(1)$ of the two Taylor coefficients in \eqref{eq:boundary-taylor}. Note that the form $Q$ is not positive on all of $H^1(B_1)$: it vanishes on the radial direction $\psi$ and on the three translations $\pt_i\psi$. The following lemma retains these directions and quantifies their defect.

Before we state the lemma, let us introduce the necessary notation. For $\ell\in\N\cup\{0\}$ we denote by $\mathcal{S}_\ell$ the space of spherical harmonics of degree $\ell$ on $\pt B_1$ and by $\{Y_\ell^m\}_{m=1}^{2\ell+1}$ an orthonormal basis of $\mathcal{S}_\ell$ in $L^2(\pt B_1)$, chosen so that $Y_0^1=(4\pi)^{-1/2}$ and $Y_1^m(\theta)=\sqrt{3/(4\pi)}\,\theta_m$ for $m=1,2,3$. The family $\{Y_\ell^m\}_{\ell,m}$ is an orthonormal basis of $L^2(\pt B_1)$, and $-\Delta_{\pt B_1}Y_\ell^m=\ell(\ell+1)Y_\ell^m$. For $z\in H^1(B_1)$ and $0<\rho\leq1$ we write
	\[
		z_\ell^m(\rho) \defeq \int_{\pt B_1} z(\rho\theta)\,Y_\ell^m(\theta)\dd\Htwo(\theta)
	\]
for the Fourier coefficient of $z$ on the sphere of radius $\rho$, so that $z(\rho\theta)=\sum_{\ell,m}z_\ell^m(\rho)Y_\ell^m(\theta)$ in $L^2(\pt B_1)$.

\begin{lemma}\label{lem:robin}
There exist $\kappa>0$ and $C>0$ such that, for every $z\in H^1(B_1)$,
	\begin{equation}\label{eq:robin}
		Q(z) \geq \kappa\norm{z}_{H^1(B_1)}^2 - C\bigg(\Big|\int_{B_1}\psi z\dd x\Big|^2 + \Big|\int_{\pt B_1}z\dd\Htwo(\theta)\Big|^2 + \Big|\int_{\pt B_1}z\,\theta\dd\Htwo(\theta)\Big|^2\bigg).
	\end{equation}
\end{lemma}

\begin{proof}
For a function $a$ on $(0,1]$ and $\ell\in\N\cup\{0\}$ set
	\begin{align*}
		Q_\ell(a) &\defeq \int_0^1\big(\rho^2|a'|^2+\big(\ell(\ell+1)-\pi^2\rho^2\big)a^2\big)\dd\rho + 2a(1)^2, \\
		N_\ell(a) &\defeq \int_0^1\big(\rho^2|a'|^2+\big(\ell(\ell+1)+\rho^2\big)a^2\big)\dd\rho.
	\end{align*}
By the orthonormality of the $Y_\ell^m$ and the identity $-\Delta_{\pt B_1}Y_\ell^m=\ell(\ell+1)Y_\ell^m$, the Dirichlet integral, the $L^2$ norm, and the trace norm of $z$ decouple over the modes, and
	\begin{equation}\label{eq:Q-modes}
		Q(z) = \sum_{\ell=0}^\infty\sum_{m=1}^{2\ell+1}Q_\ell(z_\ell^m), \qquad \norm{z}_{H^1(B_1)}^2 = \sum_{\ell=0}^\infty\sum_{m=1}^{2\ell+1}N_\ell(z_\ell^m).
	\end{equation}
Thus $Q_\ell(a)=Q(aY_\ell^m)$ and $N_\ell(a)=\norm{aY_\ell^m}_{H^1(B_1)}^2$ for every $m$, and \eqref{eq:Q-modes} says that the modes are orthogonal for both forms. These identities hold for smooth $z$ and extend to $H^1(B_1)$ by density. In particular every coefficient $z_\ell^m$ belongs to $H^1_{\rm loc}((0,1])$ and satisfies $N_\ell(z_\ell^m)<\infty$.

Since $Y_0^1=(4\pi)^{-1/2}$, $Y_1^m(\theta)=\sqrt{3/(4\pi)}\,\theta_m$, and $\psi$ is radial, the three quantities on the right-hand side of \eqref{eq:robin} are
	\[
		\int_{\pt B_1}z\dd\Htwo(\theta) = \sqrt{4\pi}\,z_0^1(1), \quad
		\int_{\pt B_1}z\,\theta_m\dd\Htwo(\theta) = \sqrt{\tfrac{4\pi}3}\,z_1^m(1), \quad
		\int_{B_1}\psi z\dd x = \sqrt{4\pi}\int_0^1\rho^2\psi\,z_0^1\dd\rho.
	\]
We first prove the lemma when these three quantities vanish. By \eqref{eq:Q-modes} it suffices to show that there is $c>0$ such that $Q_\ell(a)\geq cN_\ell(a)$ for every $\ell$ and every $a\in H^1_{\rm loc}((0,1])$ with $N_\ell(a)<\infty$, subject to $a(1)=0$ and $\int_0^1\rho^2\psi a\dd\rho=0$ when $\ell=0$, and to $a(1)=0$ when $\ell=1$.

\smallskip

\emph{Degree zero.} Let $a(1)=0$ and $\int_0^1\rho^2\psi a\dd\rho=0$. The radial function $a(|x|)$ belongs to $H^1_0(B_1)$ and is $L^2$-orthogonal to $\psi$. The Dirichlet eigenvalues of $B_1$ on radial functions are $(k\pi)^2$, $k\in\N$, with eigenfunctions $\sin(k\pi\rho)/\rho$, so the min-max principle in the radial sector gives $\int_0^1\rho^2|a'|^2\dd\rho\geq4\pi^2\int_0^1\rho^2a^2\dd\rho$. Hence,
	\[
		Q_0(a) = \int_0^1\rho^2|a'|^2\dd\rho - \pi^2\int_0^1\rho^2a^2\dd\rho \geq \frac34\int_0^1\rho^2|a'|^2\dd\rho \geq \frac{3\pi^2}{1+4\pi^2}\,N_0(a).
	\]
	
\smallskip

\emph{Degree one.} Let $a(1)=0$. Again $a(|x|)\in H^1_0(B_1)$, so the first Dirichlet eigenvalue of $B_1$ gives $\int_0^1\rho^2|a'|^2\dd\rho\geq\pi^2\int_0^1\rho^2a^2\dd\rho$, while trivially $2\int_0^1a^2\dd\rho\geq2\int_0^1\rho^2a^2\dd\rho$. Adding, we obtain
	\[
		Q_1(a) = \int_0^1\rho^2|a'|^2\dd\rho + \int_0^1(2-\pi^2\rho^2)a^2\dd\rho \geq 2\int_0^1\rho^2a^2\dd\rho,
	\]
and since $N_1(a)=Q_1(a)+(\pi^2+1)\int_0^1\rho^2a^2\dd\rho$ this yields $Q_1(a)\geq\frac{2}{\pi^2+3}N_1(a)$.

\smallskip

\emph{Degree at least two.} The hypotheses of the lemma do not restrict these modes, and in particular the value $a(1)$ is arbitrary. Since $\ell(\ell+1)<\pi^2$ for $\ell=2$, the angular term alone does not control the negative term in $Q_\ell$, and the coercivity has to come from the boundary term $2a(1)^2$.

To this end, let $p\defeq\psi'$ denote the radial profile of the translation modes $\pt_i\psi(x)=p(\rho)\theta_i$. Differentiating the radial equation $-(\rho^2\psi')'=\pi^2\rho^2\psi$ shows that
	\begin{equation}\label{eq:p-equation}
		-(\rho^2p')' + (2-\pi^2\rho^2)p = 0 \quad\text{on }(0,1],
	\end{equation}
and evaluating at $\rho=1$, where $\psi$ vanishes, yields $p'(1)=-2p(1)$. Moreover, by the explicit form of $\psi$ given in \eqref{eq:first-eigenfunction}, we see that $p<0$ on $(0,1]$.

We claim that, for every $\ell\geq0$ and every $a\in H^1_{\rm loc}((0,1])$ with $N_\ell(a)<\infty$,
	\begin{equation}\label{eq:factorization}
		Q_\ell(a) = \int_0^1\rho^2p^2\Big|\Big(\frac ap\Big)'\Big|^2\dd\rho + \big(\ell(\ell+1)-2\big)\int_0^1a^2\dd\rho.
	\end{equation}
To see this set $v\defeq a/p$, which is well-defined and belongs to $H^1_{\loc}((0,1])$ since $p$ does not vanish there. Then $\rho^2|a'|^2=\rho^2p^2|v'|^2+\rho^2(p')^2v^2+\rho^2pp'(v^2)'$, and integrating the last term by parts on $(\e,1)$, with $(\rho^2pp')'=(\rho^2p')'p+\rho^2(p')^2$, we get
	\[
		\int_\e^1\rho^2|a'|^2\dd\rho = \int_\e^1\rho^2p^2|v'|^2\dd\rho + \big[\rho^2pp'v^2\big]_\e^1 - \int_\e^1(\rho^2p')'p\,v^2\dd\rho.
	\]
By \eqref{eq:p-equation} the last integral equals $\int_\e^1(2-\pi^2\rho^2)a^2\dd\rho$. The boundary term at $\rho=1$ equals $p'(1)p(1)^{-1}a(1)^2=-2a(1)^2$, which cancels the term $2a(1)^2$ in $Q_\ell(a)$. The boundary term at $\rho=\e$ equals $\e^2p'(\e)p(\e)^{-1}a(\e)^2$, and $\e^2p'(\e)p(\e)^{-1}=O(\e)$ as $\e\to0$. Since $\int_0^1a^2\dd\rho<\infty$, there is a sequence $\e_k\to0$ along which $\e_ka(\e_k)^2\to0$. Letting $k\to\infty$ in the identity with $\e=\e_k$ proves \eqref{eq:factorization}.

Now for $\ell\geq2$ we have $\ell(\ell+1)-2\geq4$, and the first term in \eqref{eq:factorization} is nonnegative, so that
	\[
		Q_\ell(a) \geq 4\int_0^1a^2\dd\rho \geq 4\int_0^1\rho^2a^2\dd\rho.
	\]
Since $N_\ell(a)\leq Q_\ell(a)+(\pi^2+1)\int_0^1\rho^2a^2\dd\rho$ this yields $Q_\ell(a)\geq\frac{4}{\pi^2+5}N_\ell(a)$.

Summing the three estimates over the modes in \eqref{eq:Q-modes} proves 
	\[
		Q(z)\geq\kappa_0\norm{z}_{H^1(B_1)}^2,
	\]
with $\kappa_0\defeq\frac{2}{\pi^2+3}$, whenever the three quantities in \eqref{eq:robin} vanish.

In the general case we subtract from $z$ a combination of the low modes. Namely, we set
	\[
		\lambda_0 \defeq \frac1{4\pi}\int_{\pt B_1}z\dd\Htwo(\theta), \qquad
		\lambda_1 \defeq \frac3{4\pi}\int_{\pt B_1}z\,\theta\dd\Htwo(\theta), \qquad
		\lambda_2 \defeq \int_{B_1}\psi z\dd x - \lambda_0\int_{B_1}\psi\dd x,
	\]
and define
	\[
		w\defeq\lambda_0+\lambda_1\cdot x+\lambda_2\psi \qquad \text{and} \qquad \zeta\defeq z-w.
	\]
Since $\psi$ vanishes on $\pt B_1$, $\int_{\pt B_1}\theta\dd\Htwo(\theta)=0$, $\int_{\pt B_1}\theta\otimes\theta\dd\Htwo(\theta)=\frac{4\pi}3\,\id$, $\int_{B_1}\psi\,x\dd x=0$, and $\int_{B_1}\psi^2\dd x=1$, it follows that the three quantities in \eqref{eq:robin} vanish for $\zeta$. Hence, $Q(\zeta)\geq\kappa_0\norm{\zeta}_{H^1(B_1)}^2$. Moreover
	\begin{multline*}
		\norm{w}_{H^1(B_1)}^2 \leq C\big(|\lambda_0|^2+|\lambda_1|^2+|\lambda_2|^2\big) \\ \leq C\bigg(\Big|\int_{B_1}\psi z\dd x\Big|^2 + \Big|\int_{\pt B_1}z\dd\Htwo(\theta)\Big|^2 + \Big|\int_{\pt B_1}z\,\theta\dd\Htwo(\theta)\Big|^2\bigg).
	\end{multline*}
Since $Q$ is a bounded quadratic form on $H^1(B_1)$, Young's inequality gives
	\begin{multline*}
		Q(z) = Q(\zeta) + 2\int_{B_1}\big(\nabla\zeta\cdot\nabla w-\pi^2\zeta w\big)\dd x + 4\int_{\pt B_1}\zeta w\dd\Htwo(\theta) + Q(w)
		\\ \geq \frac{\kappa_0}2\norm{\zeta}_{H^1(B_1)}^2 - C\norm{w}_{H^1(B_1)}^2.
	\end{multline*}
Together with $\norm{z}_{H^1(B_1)}^2\leq2\norm{\zeta}_{H^1(B_1)}^2+2\norm{w}_{H^1(B_1)}^2$ this proves \eqref{eq:robin} with $\kappa=\kappa_0/4$.
\end{proof}

\bigskip
%%%%%%%%%%%%%%%%%%%%%%%%%%%%%%%%%%%%%%%%%%%%%%%%%%%%%%%%%%%%%%%%%%%%%%%%%%%%%%
%%%%%%%%%%%%%%%%%%%%%%%%%%%%%%%%%%%%%%%%%%%%%%%%%%%%%%%%%%%%%%%%%%%%%%%%%%%%%%
\section{Stability of the ball}\label{sec:proof}
%%%%%%%%%%%%%%%%%%%%%%%%%%%%%%%%%%%%%%%%%%%%%%%%%%%%%%%%%%%%%%%%%%%%%%%%%%%%%%
%%%%%%%%%%%%%%%%%%%%%%%%%%%%%%%%%%%%%%%%%%%%%%%%%%%%%%%%%%%%%%%%%%%%%%%%%%%%%%

In this section we prove Theorem~\ref{thm:local} by transporting the comparison identity of Lemma~\ref{lem:comparison} to the fixed ball through the diffeomorphism of Lemma~\ref{lem:extension} and bounding the resulting quadratic form from below by the Robin form \eqref{eq:Q}. The decomposition of the state is chosen so that the three defects in Lemma~\ref{lem:robin} are controlled by the volume and barycenter constraints and by the orthogonality in \eqref{eq:orthogonality}, and are therefore absorbed into the coercive term with no additive error in $q$. 

\bigskip

\begin{proof}[Proof of Theorem~\ref{thm:local}] Throughout the proof $q$ and $\vphi$ are as in \eqref{eq:constraints}, and we write $\Om\defeq\Om_\vphi$ and $T\defeq T_\vphi$. Since $\norm{\vphi}_{C^1}\leq\delta$, $\Om$ is a Lipschitz domain contained in $B_2$, and Lemma~\ref{lem:comparison} applies. Reducing $\delta$ if necessary, we also assume that $-\Psi_q'\geq c_q/2$ on $[1-\delta,1+\delta]$, which is possible by \eqref{eq:boundary-taylor} and the uniform $C^2$ bound of Lemma~\ref{lem:ball}. Let $v\in H^1_0(\Om)$ with $\norm{v}_2=1$. We choose the coefficient in Lemma~\ref{lem:comparison} by orthogonal projection,
	\begin{equation}\label{eq:decomposition}
		\alpha \defeq \frac{\int_\Om v\Psi_q\dd x}{\int_\Om\Psi_q^2\dd x}, \qquad r\defeq v-\alpha\Psi_q, \qquad z\defeq r\circ T\in H^1(B_1),
	\end{equation}
so that
	\begin{equation}\label{eq:orthogonality}
		\int_\Om\Psi_qr\dd x=0, \qquad \alpha^2\int_\Om\Psi_q^2\dd x+\norm{r}_{L^2(\Om)}^2=1, \qquad z\Big|_{\pt B_1}=-\alpha\Psi_q(1+\vphi).
	\end{equation}
The last identity holds because the trace of $r$ on $\pt\Om$ is $-\alpha\Psi_q$,  $T(\theta)=(1+\vphi(\theta))\theta$, and $\Psi_q$ is radial. Since $T$ is bi-Lipschitz with constants close to one, the norms of $r$ on $\Om$ and of $z$ on $B_1$ are uniformly equivalent, in $L^2$ as well as in $H^1$.

\smallskip
\noindent\emph{Step 1: the constraint errors.} Using the exact volume and barycenter formulas
	\begin{equation}\label{eq:moments}
		|\Om_\vphi| = \frac13\int_{\pt B_1}(1+\vphi)^3\dd\Htwo(\theta), \qquad \int_{\Om_\vphi}x\dd x = \frac14\int_{\pt B_1}(1+\vphi)^4\theta\dd\Htwo(\theta)
	\end{equation}
and the constraints in \eqref{eq:constraints} we obtain
	\begin{equation}\label{eq:phi-moments}
		\begin{gathered}
		\int_{\pt B_1}\vphi\dd\Htwo(\theta) = -\int_{\pt B_1}\vphi^2\dd\Htwo(\theta) - \frac13\int_{\pt B_1}\vphi^3\dd\Htwo(\theta),\\
		\int_{\pt B_1}\vphi\,\theta\dd\Htwo(\theta) = -\frac32\int_{\pt B_1}\vphi^2\theta\dd\Htwo(\theta) - \int_{\pt B_1}\vphi^3\theta\dd\Htwo(\theta) - \frac14\int_{\pt B_1}\vphi^4\theta\dd\Htwo(\theta).
		\end{gathered}
	\end{equation}
Hence, noting that $\norm{\vphi}_\infty\leq\delta\leq1$, we get
	\begin{equation}\label{eq:low-modes}
		\Big|\int_{\pt B_1}\vphi\dd\Htwo(\theta)\Big| + \Big|\int_{\pt B_1}\vphi\,\theta\dd\Htwo(\theta)\Big| \leq C\norm{\vphi}_{L^2}^2.
	\end{equation}
Additionally, since $\int_{B_1}\Psi_q^2\dd x=1$ and, by \eqref{eq:boundary-taylor}, $|\Psi_q(1+s)|\leq C|s|$ for $|s|\leq1$, we have that
	\begin{multline}\label{eq:mass-error}
		\Big|\int_\Om\Psi_q^2\dd x-1\Big| = \Big|\int_{\pt B_1}\int_1^{1+\vphi(\theta)}\Psi_q(\rho)^2\rho^2\dd\rho\dd\Htwo(\theta)\Big| \\ \leq C\int_{\pt B_1}|\vphi|^3\dd\Htwo(\theta) \leq C\delta\norm{\vphi}_{L^2}^2.
	\end{multline}
In particular, the denominator in \eqref{eq:decomposition} lies between $1/2$ and $2$, and $|\alpha|\leq\sqrt2$.

We now show that the boundary trace of $z$ controls the graph. On $[-\delta,\delta]$ the function $s\mapsto-\Psi_q(1+s)$ vanishes at zero and has derivative at least $c_q/2$, so it is a bi-Lipschitz bijection onto its image, and its inverse extends to a Lipschitz function $G$ on $\R$ with $G(0)=0$ and Lipschitz constant at most $2/c_q$. If $\alpha\neq0$, then $\vphi=G(\alpha^{-1}z\big|_{\pt B_1})$ by \eqref{eq:orthogonality}. Since $G(0)=0$, by the Gagliardo form of the $H^{1/2}(\partial B_1)$ norm, $\norm{G\circ f}_{H^{1/2}}\leq(2/c_q)\norm{f}_{H^{1/2}}$ for every $f\in H^{1/2}(\partial B_1)$, and the trace theorem on $B_1$ then yields
	\begin{equation}\label{eq:trace-control}
		|\alpha|\norm{\vphi}_{H^{1/2}} \leq C\norm{z}_{H^1(B_1)},
	\end{equation}
which is trivial if $\alpha=0$. By \eqref{eq:boundary-taylor}, we have $z\big|_{\pt B_1}=\alpha c_q\vphi-\alpha c_q\vphi^2+O(|\alpha||\vphi|^3)$. The linear terms are controlled by \eqref{eq:low-modes} and $\int_{\pt B_1}|\vphi|^3\dd\Htwo(\theta)\leq\norm{\vphi}_\infty\norm{\vphi}_{L^2}^2$; hence, with \eqref{eq:trace-control}, we get that
	\begin{equation}\label{eq:boundary-moments}
		\Big|\int_{\pt B_1}z\dd\Htwo(\theta)\Big| + \Big|\int_{\pt B_1}z\,\theta\dd\Htwo(\theta)\Big| \leq C|\alpha|\norm{\vphi}_{L^2}^2 \leq C\delta\,|\alpha|\norm{\vphi}_{H^{1/2}} \leq C\delta\norm{z}_{H^1(B_1)}.
	\end{equation}
For the radial moment, the orthogonality in \eqref{eq:orthogonality} and changing variable through $T$ give exactly
	\[
		\int_{B_1}\psi z\dd x = \int_{B_1}\big(\psi-(\Psi_q\circ T)\det DT\big)z\dd x.
	\]
The coefficient is bounded in $L^\infty(B_1)$ by
	\[
		\norm{\psi-u_q}_\infty + \norm{\Psi_q\circ T-\Psi_q}_\infty + \norm{\Psi_q}_\infty\norm{\det DT-1}_\infty \leq C(q+\delta),
	\]
by \eqref{eq:ball-estimates}, the Lipschitz bound on $\Psi_q$, and \eqref{eq:map-bounds}. Therefore
	\begin{equation}\label{eq:radial-moment}
		\Big|\int_{B_1}\psi z\dd x\Big| \leq C(q+\delta)\norm{z}_{L^2(B_1)}.
	\end{equation}
Thus the three defects in \eqref{eq:robin} are small multiples of $\norm{z}_{H^1(B_1)}$, with no additive error depending only on $q$. This is the point which distinguishes the present decomposition from the one in Remark~\ref{rem:dirichlet}.

\smallskip
\noindent\emph{Step 2: the comparison form.} Since $\Psi_q$ is radial, $\pt_\nu\Psi_q=\Psi_q'(|x|)\,\theta\cdot\nu$ on $\pt\Om$, and for the radial graph $x=(1+\vphi(\theta))\theta$ one has $\theta\cdot\nu\dd\Htwo(x)=(1+\vphi)^2\dd\Htwo(\theta)$. Hence,
	\begin{equation}\label{eq:boundary-integral}
		\int_{\pt\Om}\Psi_q\,\pt_\nu\Psi_q\dd\Htwo(x) = \int_{\pt B_1}\Psi_q(1+\vphi)\Psi_q'(1+\vphi)(1+\vphi)^2\dd\Htwo(\theta).
	\end{equation}
Using \eqref{eq:boundary-taylor} and $\Psi_q'(1+s)=-c_q+2c_qs+O(s^2)$, we get
	\[
		\Psi_q(1+s)\Psi_q'(1+s)(1+s)^2 = c_q^2s - c_q^2s^2 + O(|s|^3).
	\]
Inserting this in \eqref{eq:boundary-integral}, eliminating $\int_{\pt B_1}\vphi\dd\Htwo(\theta)$ through the volume identity in \eqref{eq:phi-moments}, and then expressing $\int_{\pt B_1}\vphi^2\dd\Htwo(\theta)$ through the trace in \eqref{eq:orthogonality}, we obtain
	\begin{equation}\label{eq:boundary-comparison}
		\begin{aligned}
		-\alpha^2\int_{\pt\Om}\Psi_q\,\pt_\nu\Psi_q\dd\Htwo(x) &= 2\alpha^2c_q^2\int_{\pt B_1}\vphi^2\dd\Htwo(\theta) + O\big(\delta\alpha^2\norm{\vphi}_{L^2}^2\big)\\
		&= 2\int_{\pt B_1}z^2\dd\Htwo(\theta) + O\big(\delta\norm{z}_{H^1(B_1)}^2\big),
		\end{aligned}
	\end{equation}
where in the last step we used $z^2\big|_{\pt B_1}=\alpha^2c_q^2\vphi^2+O(\alpha^2|\vphi|^3)$ and \eqref{eq:trace-control}. Now in the bulk term of \eqref{eq:comparison} we change variables through $T$. The coefficients become $(\det DT)\,DT^{-1}DT^{-T}$ and $(qV_q\circ T-\mu_q)\det DT$, which differ in $L^\infty(B_1)$ from the identity matrix and from $-\pi^2$, respectively, by at most $C(q+\delta)$, thanks to \eqref{eq:ball-estimates} and \eqref{eq:map-bounds}. Therefore 
	\begin{equation}\label{eq:bulk-comparison}
		\Big|\int_\Om\big(|\nabla r|^2+(qV_q-\mu_q)r^2\big)\dd x - \int_{B_1}\big(|\nabla z|^2-\pi^2z^2\big)\dd x\Big| \leq C(q+\delta)\norm{z}_{H^1(B_1)}^2.
	\end{equation}
Combining \eqref{eq:comparison}, \eqref{eq:boundary-comparison}, and \eqref{eq:bulk-comparison} with Lemma~\ref{lem:robin}, and absorbing the defects through \eqref{eq:boundary-moments} and \eqref{eq:radial-moment}, we arrive at
	\begin{equation}\label{eq:coercive-comparison}
		\E(v,\Om)-\en(B_1) \geq \kappa\norm{z}_{H^1(B_1)}^2 + q\nl(v^2-u_q^2,v^2-u_q^2)
	\end{equation}
provided $q_0$ and $\delta$ are small enough, with $\kappa$ half the constant of Lemma~\ref{lem:robin}. Note that no smallness of $\E(v,\Om)-\en(B_1)$ has been assumed in deriving \eqref{eq:coercive-comparison}.

\smallskip
\noindent\emph{Step 3: recovering the shape and the state.} By \eqref{eq:orthogonality} and \eqref{eq:mass-error}, $1\leq2\alpha^2+\norm{r}_{L^2(\Om)}^2$. Multiplying both sides by $\norm{\vphi}_{H^{1/2}}^2$, and using \eqref{eq:trace-control} along with the estimate $\norm{\vphi}_{H^{1/2}}^2\leq C\delta^2$, we obtain
	\begin{equation}\label{eq:shape-control}
		\norm{\vphi}_{H^{1/2}}^2 \leq 2\alpha^2\norm{\vphi}_{H^{1/2}}^2 + C\delta^2\norm{r}_{L^2(\Om)}^2 \leq C\norm{z}_{H^1(B_1)}^2.
	\end{equation}
This estimate also covers the case $\alpha=0$. Next, since $\alpha^2\int_\Om\Psi_q^2\dd x=1-\norm{r}_2^2$ and $1-\sqrt{1-s}\leq s$ for $0\leq s\leq1$, we have $0\leq1-|\alpha|\big(\int_\Om\Psi_q^2\dd x\big)^{1/2}\leq\norm{r}_{L^2(\Om)}^2$. Hence,
	\[
		\big|1-|\alpha|\big| \leq \norm{r}_{L^2(\Om)}^2 + |\alpha|\Big|\Big(\int_\Om\Psi_q^2\dd x\Big)^{1/2}-1\Big| \leq \norm{r}_{L^2(\Om)}^2 + 2\Big|\int_\Om\Psi_q^2\dd x-1\Big|.
	\]
Thus by \eqref{eq:mass-error} and \eqref{eq:shape-control} we get
	\begin{equation}\label{eq:alpha-control}
		(1-|\alpha|)^2 \leq C\norm{z}_{H^1(B_1)}^2.
	\end{equation}
Moreover, since $\Psi_q=u_q$ on $B_1$ and $\Psi_q$ is bounded in $C^2(\ol{B}_2)$, the mean value theorem applied to $\Psi_q$ and to $\nabla\Psi_q$, together with \eqref{eq:map-bounds}, yields
	\begin{equation}\label{eq:transport-reference}
		\norm{\Psi_q\circ T-u_q}_{H^1(B_1)} \leq C\norm{T-\id}_{H^1(B_1)} \leq C\norm{\vphi}_{H^{1/2}}.
	\end{equation}
We choose $\sigma\in\{-1,1\}$ with $\sigma\alpha=|\alpha|$, and $\sigma=1$ if $\alpha=0$. Since
	\[
		v\circ T-\sigma u_q = z + (\alpha-\sigma)\,\Psi_q\circ T + \sigma\big(\Psi_q\circ T-u_q\big),
	\]
the estimates \eqref{eq:shape-control}--\eqref{eq:transport-reference} bound $\norm{v\circ T-\sigma u_q}_{H^1(B_1)}^2$ by $C\norm{z}_{H^1(B_1)}^2$. Together with \eqref{eq:coercive-comparison} and \eqref{eq:shape-control}, this proves \eqref{eq:joint}.

Suppose now that $v\geq0$. If $\alpha\geq0$, then $\sigma=1$ by the choice above. Let $\alpha<0$. Since $v\geq0$ and $\Psi_q\geq-(\Psi_q)_-$, the definition of $\alpha$, the Cauchy--Schwarz inequality with $\norm{v}_2=1$, and the lower bound on the denominator in \eqref{eq:decomposition} imply
	\[
		|\alpha| = -\frac{\int_\Om v\Psi_q\dd x}{\int_\Om\Psi_q^2\dd x} \leq 2\norm{(\Psi_q)_-}_{L^2(\Om)}.
	\]
By Lemma~\ref{lem:ball}(i), $\Psi_q=u_q>0$ in $B_1$, so the negative part of $\Psi_q$ on $\Om$ is supported in $\Om\setminus\ol{B}_1$, where $|\Psi_q|\leq C|\vphi|$ by \eqref{eq:boundary-taylor}. Hence,
	\[
		\norm{(\Psi_q)_-}_{L^2(\Om)}^2 \leq C\int_{\pt B_1}|\vphi|^3\dd\Htwo \leq C\delta^3,
	\]
and $|\alpha|\leq C\delta^{3/2}$. For $\delta$ sufficiently small this estimate, combined with \eqref{eq:orthogonality}, yields $\norm{r}_{L^2(\Om)}^2=1-\alpha^2\int_\Om\Psi_q^2\dd x\geq1/2$; hence, $\norm{z}_{L^2(B_1)}\geq c>0$. Since $v\circ T-u_q=z+\alpha\Psi_q\circ T-u_q$ and $|\alpha|\leq\sqrt2$, we then have $\norm{v\circ T-u_q}_{H^1(B_1)}\leq\norm{z}_{H^1(B_1)}+C\leq C'\norm{z}_{H^1(B_1)}$. Thus $\sigma=1$ is admissible for nonnegative states.

\smallskip
\noindent\emph{Step 4: the upper bound.} Consider the trial state 
	\[
		v\defeq \frac{u_q\circ T^{-1}}{\norm{u_q\circ T^{-1}}_{L^2(\Om)}},
	\] 
where $u_q$ is extended by zero outside $B_1$. Since $\norm{u_q\circ T^{-1}}_{L^2(\Om)}^2=\int_{B_1}u_q^2\det DT\dd x$, $u_q$ is bounded, and $|\det DT-1|\leq C|DT-\id|$ pointwise, the normalization factor differs from one by at most $C\norm{DT-\id}_{L^2(B_1)}\leq C\norm{\vphi}_{H^{1/2}}$, thanks to \eqref{eq:map-bounds}. Moreover, the zero extension of $u_q$ is Lipschitz on $\R^3$, uniformly in $q$ as it vanishes on $\pt B_1$, its gradient is bounded in $B_1$ by \eqref{eq:ball-estimates}, and $B_1$ is convex. Writing $T-\id=\xi$, so that $T^{-1}(x)-x=-\xi(T^{-1}(x))$, and changing variables through $T$, we get
	\begin{equation}\label{eq:lipschitz-transport}
		\norm{u_q\circ T^{-1}-u_q}_{L^2(\R^3)} \leq C\norm{\xi\circ T^{-1}}_{L^2(\R^3)} \leq C\norm{\xi}_{L^2(\R^3)} \leq C\norm{\vphi}_{H^{1/2}}.
	\end{equation}
Together with \eqref{eq:transport-reference} and the bound on the normalization factor, this gives
	\begin{equation}\label{eq:trial-bounds}
		\norm{(v-\Psi_q)\circ T}_{H^1(B_1)} + \norm{v-u_q}_{L^2(\R^3)} \leq C\norm{\vphi}_{H^{1/2}}.
	\end{equation}
	
We now apply Lemma~\ref{lem:comparison} to the trial state $v$ with $\alpha=1$, so that $r=v-\Psi_q$ in $\Om$, and estimate the three terms on the right-hand side of \eqref{eq:comparison}. The first term is bounded by $C\norm{(v-\Psi_q)\circ T}_{H^1(B_1)}^2$, because $qV_q-\mu_q$ is bounded and $T$ is bi-Lipschitz. The boundary term equals $2c_q^2\int_{\pt B_1}\vphi^2\dd\Htwo+O(\delta\norm{\vphi}_{L^2}^2)$, by the computation leading to \eqref{eq:boundary-comparison} with $\alpha=1$; hence it is of order $O(\norm{\vphi}_{L^2}^2)$. For the Coulomb term, the Hardy--Littlewood--Sobolev inequality and H\"older's inequality give
	\[
		\nl(v^2-u_q^2,v^2-u_q^2) \leq C\norm{(v-u_q)(v+u_q)}_{6/5}^2 \leq C\norm{v-u_q}_2^2\norm{v+u_q}_3^2 \leq C\norm{\vphi}_{H^{1/2}}^2,
	\]
since $v$ and $u_q$ have uniformly bounded $H^1$ norms and \eqref{eq:trial-bounds} holds. Adding the three bounds and using \eqref{eq:trial-bounds} once more for the first term, we obtain 
	\[
		\E(v,\Om)-\en(B_1)\leq C\norm{\vphi}_{H^{1/2}}^2,
	\]
and since $v$ is admissible in the definition of $\en(\Om)$, this proves the upper bound in \eqref{eq:two-sided}. The lower bound follows by taking the infimum over $v$ in \eqref{eq:joint}, and therefore the proof of Theorem~\ref{thm:local} is complete.
\end{proof}	

\medskip

We now justify the variant of \eqref{eq:joint} stated in Remark~\ref{rem:L2}, in which $\norm{v-\sigma u_q}_2$ replaces the transported distance, with $v$ and $u_q$ extended by zero. For an arbitrary normalized $v\in H^1_0(\Om)$ and either sign $\sigma\in\{-1,1\}$,
	\[
		v-\sigma u_q = \big(v-\sigma\,u_q\circ T^{-1}\big) + \sigma\big(u_q\circ T^{-1}-u_q\big) \qquad\text{on }\R^3.
	\]
The first term is supported in $\Om$, and changing variables through $T$ gives
	\[
		\norm{v-\sigma u_q\circ T^{-1}}_{L^2(\Om)}\leq C\norm{v\circ T-\sigma u_q}_{L^2(B_1)}.
	\] 
The second term is bounded by \eqref{eq:lipschitz-transport}. Hence 	\[
	\norm{v-\sigma u_q}_2^2\leq C\big(\norm{v\circ T-\sigma u_q}_{H^1(B_1)}^2+\norm{\vphi}_{H^{1/2}}^2\big)
	\]
for both signs, and minimizing over $\sigma$ and using \eqref{eq:joint} proves the claim.

\begin{remark}[Sharpness of the order]\label{rem:sharp}
The quadratic power in \eqref{eq:two-sided} cannot be improved: for a nonzero spherical harmonic $h$ of even degree $\ell\geq2$, the graph $th+a(t)$, with the constant $a(t)=O(t^2)$ chosen to restore the volume, has barycenter zero by central symmetry, and \eqref{eq:two-sided} shows that its deficit is of order exactly $t^2$. The Sobolev exponent $1/2$ cannot be raised either. Suppose that $\en(\Om_\vphi)-\en(B_1)\geq\kappa'\norm{\vphi}_{H^s}^2$ held under \eqref{eq:constraints} for some $s>1/2$. Then the upper bound in \eqref{eq:two-sided}, applied to the graphs just described, would give $\norm{\vphi}_{H^s}^2\leq C\norm{\vphi}_{H^{1/2}}^2$ for every even $\ell$ and every small $t$. This is impossible, since on a spherical harmonic of degree $\ell$ the ratio $\norm{h}_{H^s}^2/\norm{h}_{H^{1/2}}^2$ is of order $\ell^{2s-1}$ (up to $O(t)$ corrections), which is unbounded in $\ell$.
\end{remark}

%%%%%%%%%%%%%%%%%%%%%%%%%%%%%%%%%%%%%%%%%%%%%%%%%%%%%%%%%%%%%%%%%%%%%%%%%%%%%%
%%%%%%%%%%%%%%%%%%%%%%%%%%%%%%%%%%%%%%%%%%%%%%%%%%%%%%%%%%%%%%%%%%%%%%%%%%%%%%
\section{Minimality of the ball}\label{sec:global}
%%%%%%%%%%%%%%%%%%%%%%%%%%%%%%%%%%%%%%%%%%%%%%%%%%%%%%%%%%%%%%%%%%%%%%%%%%%%%%
%%%%%%%%%%%%%%%%%%%%%%%%%%%%%%%%%%%%%%%%%%%%%%%%%%%%%%%%%%%%%%%%%%%%%%%%%%%%%%

In this section we prove Theorem~\ref{thm:global} by combining the near-sphericity of minimizers established in \cite[Theorem 1.2]{MMR} with the lower bound in \eqref{eq:two-sided}. The estimate \eqref{eq:two-sided} identifies a minimizer with a translate of $B_1$ only up to a null set, whereas admissible sets are identified up to capacity zero, and the passage from the former to the latter is the second part of the proof.

\bigskip

\begin{proof}[Proof of Theorem~\ref{thm:global}]
The case $q=0$ is the Faber--Krahn inequality, with its equality case up to capacity zero. Let $q>0$. By \cite[Theorem 1.2]{MMR} (cf. \eqref{eq:MMR} above), for every $\e>0$ there is $q^*(\e)>0$ such that for $0<q\leq q^*(\e)$ a minimizer exists and every minimizer is, after a translation and up to a null set, a radial graph $\Om_\vphi$ with $\norm{\vphi}_{C^2(\pt B_1)}\leq\e$. Let $\Om$ be a minimizer and let $u$ be a nonnegative optimal state on $\Om$. The regularity statement of \cite{MMR} concerns the positivity set $\{u>0\}$ (cf.\ \cite[Lemmas 3.6 and 3.7]{MMR}), which we call the regular representative of $\Om$. Note that $\{u>0\}$ coincides with $\Om$ up to a null set for every minimizer. Indeed, $\{u>0\}\subset\Om$ has the same energy as $\Om$, because $u$ is admissible on $\{u>0\}$ and the energy is monotone with respect to inclusion (cf.\ \cite[Lemma 2.7]{MMR}). If $|\{u>0\}|<|B_1|$, then dilating $\{u>0\}$ to volume $|B_1|$ would strictly decrease the energy, since for $t>1$ the rescaled state $u_t(x)\defeq t^{-3/2}u(x/t)$ on $t\{u>0\}$ satisfies
	\[
		\E\big(u_t,t\{u>0\}\big) = t^{-2}\int_{\{u>0\}}|\nabla u|^2\dd x + t^{-1}q\nl(u^2,u^2) < \E\big(u,\{u>0\}\big),
	\]
contradicting minimality. Hence $|\{u>0\}|=|B_1|$, and $\{u>0\}$ coincides with $\Om$ up to a null set.

We translate once more so that the barycenter is at the origin. This preserves the smallness of the radial graph. By \eqref{eq:moments}, the displacement of the barycenter is bounded by $C\norm{\vphi}_{C^0}$. A sufficiently small translation of a $C^2$-small radial graph over the unit sphere is again a radial graph over the unit sphere, with $C^2$ norm bounded by $C\norm{\vphi}_{C^2}$, by the implicit function theorem.

We now choose $\e$ sufficiently small so that the centered graph satisfies $\norm{\vphi}_X\leq\delta$, which is possible since $C^2(\pt B_1)\iarrow X$, and we replace $q_0$ by $\min\{q^*(\e),q_0\}$. Then every minimizer is, after translations, a set $\Om_\vphi$ satisfying \eqref{eq:constraints}. Since $B_1$ is admissible, $\en(\Om_\vphi)\leq\en(B_1)$, and the lower bound in \eqref{eq:two-sided} forces $\vphi=0$. Hence the regular representative of every minimizer is a translate of $B_1$, and since a minimizer exists, $e_q(|B_1|)=\en(B_1)$. The local comparison requires neither criticality nor an Euler--Lagrange equation on the varying domain.

It remains to justify the identification of admissible sets up to capacity zero in the equality case, which does not follow from equality almost everywhere alone. Let $\Om$ be an admissible quasi-open set of volume $|B_1|$ with $\en(\Om)=\en(B_1)$. By the above, after a translation, $\Om=B_1$ almost everywhere. Every $w\in H^1_0(\Om)$ then vanishes almost everywhere outside $B_1$, and the zero-extension characterization of $H^1_0$ on the Lipschitz domain $B_1$ gives $w\in H^1_0(B_1)$. Thus $H^1_0(\Om)$ is a closed subspace of $H^1_0(B_1)$. Let $u_n\in H^1_0(\Om)$ be a nonnegative normalized minimizing sequence for $\en(\Om)$. It is bounded in $H^1(B_1)$, so a subsequence converges weakly in $H^1(B_1)$ and strongly in $L^2(B_1)$ to a nonnegative normalized function $u$, which belongs to $H^1_0(\Om)$ since closed subspaces are weakly closed. The Dirichlet integral is weakly lower semicontinuous and the Coulomb term is continuous along the sequence by \eqref{eq:D-lip}, so that $\E(u,B_1)\leq\en(\Om)=\en(B_1)$. Hence $u$ is a minimizing state on the ball, and $u=u_q$ by Lemma~\ref{lem:ball}. In particular $u_q\in H^1_0(\Om)$. 

By definition of $H^1_0(\Om)$, the quasi-continuous representative of $u_q$ vanishes quasi-everywhere outside $\Om$, and since $u_q$ is continuous and positive on $B_1$, we obtain $\capa(B_1\setminus\Om)=0$. Conversely, we write $\Om$, up to capacity zero, as the positivity set of the quasi-continuous representative of a nonnegative $w\in H^1(\R^3)$, which is possible for quasi-open sets of finite measure by \cite[Section 2.3]{MMR} and exhaustion. Since $\Om=B_1$ almost everywhere, $w$ vanishes almost everywhere outside $B_1$, hence $w\in H^1_0(B_1)$ and its quasi-continuous representative vanishes quasi-everywhere outside $B_1$. This gives $\capa(\Om\setminus B_1)=0$. The two inclusions show that $\Om=B_1$ up to capacity zero. Finally, if $\Om$ coincides with a translate of $B_1$ up to capacity zero, then the two sets have the same Sobolev space and the same energy, which proves the converse.
\end{proof}

\medskip

\begin{remark}[Asymmetry and other volumes]\label{rem:scaling}
For an admissible set $\Om$ let
	\[
		\mathcal{A}(\Om)\defeq\inf_{z\in\R^3}\frac{|\Om\triangle(B_1+z)|}{|B_1|}
	\]
denote its Fraenkel asymmetry. Since 
	\[
		|\Om_\vphi\triangle B_1|=\frac13\int_{\pt B_1}|(1+\vphi)^3-1|\dd\Htwo\leq C\norm{\vphi}_{L^1}\leq C\norm{\vphi}_{H^{1/2}},
	\]
after adjusting $\kappa$, Theorem~\ref{thm:local} yields the local asymmetry estimate
	\[
		\en(\Om_\vphi)-\en(B_1)\geq\kappa\,\mathcal{A}(\Om_\vphi)^2
	\]
under the constraints \eqref{eq:constraints}. 

The results extend to other volumes by scaling. At prescribed volume $m$ we set the radius $R\defeq(m/|B_1|)^{1/3}$. For $u_R(x)\defeq R^{-3/2}u(x/R)$ the $L^2$ normalization is preserved and  we have $\E(u_R,R\Om)=R^{-2}\mathcal{E}_{qR}(u,\Om)$. Hence, $\en(R\Om)=R^{-2}\mathcal{E}_{qR}(\Om)$ and $e_q(m)=R^{-2}e_{qR}(|B_1|)$. Consequently, Theorem~\ref{thm:local} and Theorem~\ref{thm:global} hold for balls of radius $R$ under the condition $qR\leq q_0$, with the energy deficit multiplied by $R^{-2}$. The proof gives a positive threshold $q_0$ but neither an explicit nor an optimal value of it. In particular, it does not identify the transition out of the ball regime.
\end{remark}

%%%%%%%%%%%%%%%%%%%%%%%%%%%%%%%%%%%%%%%%%%%%%%%%%%%%%%%%%%%%%%%%%%%%%%%%%%%%%%
%%%%%%%%%%%%%%%%%%%%%%%%%%%%%%%%%%%%%%%%%%%%%%%%%%%%%%%%%%%%%%%%%%%%%%%%%%%%%%
\section{The Hartree--Ohta--Kawasaki model}\label{sec:quadratic}
%%%%%%%%%%%%%%%%%%%%%%%%%%%%%%%%%%%%%%%%%%%%%%%%%%%%%%%%%%%%%%%%%%%%%%%%%%%%%%
%%%%%%%%%%%%%%%%%%%%%%%%%%%%%%%%%%%%%%%%%%%%%%%%%%%%%%%%%%%%%%%%%%%%%%%%%%%%%%

In this section we show that the direct comparison applies also to the model of \cite{MMR26H}, whose nonlocal term is quadratic rather than quartic in the state, and we obtain the analogues of Theorem~\ref{thm:local} and Theorem~\ref{thm:global}. In \cite{MMR26H}, Mazzoleni, Moraschi, and Ruffini study the energy
	\begin{equation}\label{eq:quadratic-min}
		\widetilde{\mathcal{E}}_q(\Om) \defeq \inf\Big\{\widetilde{\mathcal{E}}_q(v,\Om) \colon v\in H^1_0(\Om),\ \norm{v}_2=1\Big\},
	\end{equation}
where
	\begin{equation}\label{eq:quadratic-energy}
		\widetilde{\mathcal{E}}_q(v,\Om) \defeq \int_\Om|\nabla v|^2\dd x + q\nl(v,v).
	\end{equation}
The ground states of this energy are not a priori of constant sign, and the existence and near-sphericity theorem \cite[Theorem 1.1]{MMR26H} relies on showing that they are, for small $q$. In \cite{MMR26H} the authors conjecture that balls are the unique minimizers for small $q$ and we prove this conjecture below. This conjecture has also been proved independently in \cite{MMR26R} by a second shape derivative argument.

Throughout this section we will denote by $u_q$ the positive normalized minimizer of $\widetilde{\mathcal{E}}_q(\cdot,B_1)$, and set
	\[
		\tld{V}_q\defeq Ku_q, \qquad \text{ and } \qquad \tld{\mu}_q\defeq\widetilde{\mathcal{E}}_q(B_1). 
	\]

\bigskip

\begin{corollary}\label{cor:quadratic}
There exist $q_0,\delta,\kappa,C>0$, with $q_0$ and $\delta$ possibly smaller than in Theorem~\ref{thm:local}, such that under the constraints \eqref{eq:constraints}, every $v\in H^1_0(\Om_\vphi)$ with $\norm{v}_2=1$ satisfies
	\begin{equation}\label{eq:quad-joint}
		\widetilde{\mathcal{E}}_q(v,\Om_\vphi)-\widetilde{\mathcal{E}}_q(B_1) \geq \kappa\Big(\norm{\vphi}_{H^{1/2}}^2 + \min_{\sigma\in\{-1,1\}}\norm{v\circ T_\vphi-\sigma u_q}_{H^1(B_1)}^2\Big),
	\end{equation}
and
	\begin{equation}\label{eq:quad-two-sided}
		\kappa\norm{\vphi}_{H^{1/2}}^2 \leq \widetilde{\mathcal{E}}_q(\Om_\vphi)-\widetilde{\mathcal{E}}_q(B_1) \leq C\norm{\vphi}_{H^{1/2}}^2.
	\end{equation}
Moreover, for $0\leq q\leq q_0$, balls are the unique minimizers of $\widetilde{\mathcal{E}}_q$ among admissible sets of volume $|B_1|$, up to translations and up to sets of capacity zero.
\end{corollary}

\begin{proof}
For $v\in H^1_0(B_1)$, extended by zero, \eqref{eq:coulomb-identities} gives
	\[
		\widetilde{\mathcal{E}}_q(v,B_1)=\int_{B_1}|\nabla v|^2\dd x+\frac q2\int_{B_1}v\,Kv\dd x,
	\]
i.e., $\widetilde{\mathcal{E}}_q(\cdot,B_1)$ is the quadratic form of the Dirichlet operator $-\Delta+\frac q2K$ on $B_1$, where $K$ acts on functions in $L^2(B_1)$ extended by zero. The operator $K$ is bounded on $L^2(B_1)$ and positive, since $\int_{B_1}f\,Kf\dd x=2\nl(f,f)\geq0$ by \eqref{eq:coulomb-identities}. Hence, the ground states of \eqref{eq:quadratic-energy} on $B_1$ are the first eigenfunctions of $-\Delta+\frac q2K$ with $\tld{\mu}_q$ its first eigenvalue, and
	\begin{equation}\label{eq:quad-state}
		-\Delta u_q + \frac q2\tld{V}_q = \tld{\mu}_qu_q \quad\text{in } B_1, \qquad u_q=0 \quad\text{on } \pt B_1.
	\end{equation}
	
Since $K$ is bounded and positive, the operator $-\Delta+\frac q2K$ is a bounded perturbation of the Dirichlet Laplacian on $B_1$ and lies between $-\Delta$ and $-\Delta+\frac q2\norm{K}$ in the sense of quadratic forms. By the min--max principle (see \cite[Theorem XIII.2]{RS}), its first eigenvalue lies between $\pi^2$ and $\pi^2+Cq$, and its second eigenvalue is at least the second Dirichlet eigenvalue of $B_1$. Hence, for small $q$, the first eigenvalue is simple.

We rewrite \eqref{eq:quad-state} as $(-\Delta-\pi^2)(u_q-\psi)=(\tld{\mu}_q-\pi^2)u_q-\frac q2\tld{V}_q$. Since $\norm{\tld{V}_q}_\infty\leq C\norm{u_q}_2$ by the Cauchy--Schwarz inequality and $\sup_x\int_{B_1}|x-y|^{-2}\dd y<\infty$, the argument of Lemma~\ref{lem:ball}(ii) applies with the obvious changes. After fixing the sign of $u_q$ so that $\int_{B_1}u_q\psi\dd x>0$, we get 
	\[
		\norm{u_q-\psi}_{C^2(\ol{B}_1)}\leq Cq, \qquad |\tld{\mu}_q-\pi^2|\leq Cq, \qquad \text{and} \qquad \norm{\tld{V}_q}_{C^1(\R^3)}\leq C
	\]
with bounds uniform in $q$. Since $\psi$ vanishes linearly on $\pt B_1$, there is $c_0>0$ with $\psi\geq c_0\dist(\cdot,\pt B_1)$ in $B_1$, whereas $u_q-\psi$ vanishes on $\pt B_1$ and has gradient bounded by $Cq$, so that $|u_q-\psi|\leq Cq\dist(\cdot,\pt B_1)$. Hence $u_q>0$ in $B_1$ for small $q$. Finally, $u_q\circ R$ is a positive normalized first eigenfunction for every rotation $R$, and the simplicity of the first eigenvalue gives $u_q\circ R=u_q$, so that $u_q$ is radial.

Next we continue the radial state across the sphere by solving the radial form of
	\begin{equation}\label{eq:quad-continuation}
		(-\Delta-\tld{\mu}_q)\Psi_q = -\frac q2\tld{V}_q \quad\text{in } B_2,
	\end{equation}
with $\Psi_q=u_q$ on $B_1$. As in Lemma~\ref{lem:ball}(iii), $\Psi_q$ is bounded in $C^2(\ol{B}_2)$ and in $C^3(\ol{B}_2\setminus B_{1/2})$, uniformly in $q$. Evaluating the radial equation at $\rho=1$ now gives
	\[
		\Psi_q''(1) = 2c_q + \frac q2\tld{V}_q(1), \qquad c_q = -\Psi_q'(1) = \sqrt{\pi/2}+O(q),
	\]
so that the coefficient of $s^2$ in the expansion of $\Psi_q(1+s)$ differs from $c_q$ by $O(q)$.

Let $\Om=\Om_\vphi$ be as in Section~\ref{sec:proof}, let $v\in H^1_0(\Om)$ be normalized, and choose $\alpha$, $r$, $z$ as in \eqref{eq:decomposition}. Expanding the Coulomb term about $\alpha u_q$, we get
	\[
		\nl(v,v) = \alpha\int_\Om \tld{V}_qv\dd x - \alpha^2\nl(u_q,u_q) + \nl(v-\alpha u_q,v-\alpha u_q).
	\]
Then using $\norm{v}_2=1$ and $\widetilde{\mathcal{E}}_q(B_1)=\tld{\mu}_q$, we obtain the exact identity
	\begin{equation}\label{eq:quad-comparison}
		\begin{aligned}
		\widetilde{\mathcal{E}}_q(v,\Om)-\widetilde{\mathcal{E}}_q(B_1)
		={}& \int_\Om\big(|\nabla r|^2-\tld{\mu}_qr^2\big)\dd x - \alpha^2\int_{\pt\Om}\Psi_q\,\pt_\nu\Psi_q\dd\Htwo(x)\\
		&+ \frac{q\alpha^2}{2}\Big(\int_\Om \tld{V}_q\Psi_q\dd x - \int_{B_1}\tld{V}_qu_q\dd x\Big) + q\nl(v-\alpha u_q,v-\alpha u_q).
		\end{aligned}
	\end{equation}
Indeed, inserting $v=\alpha\Psi_q+r$ into $\int_\Om(|\nabla v|^2-\tld{\mu}_qv^2)\dd x$ and integrating by parts with \eqref{eq:quad-continuation}, the pure and the mixed terms produce
	\[
		-\alpha^2\int_{\pt\Om}\Psi_q\,\pt_\nu\Psi_q\dd\Htwo(x)-\frac{q\alpha^2}2\int_\Om \tld{V}_q\Psi_q\dd x-q\alpha\int_\Om \tld{V}_qr\dd x,
	\]
while the Coulomb expansion gives
	\[
		q\alpha\int_\Om \tld{V}_qv\dd x = q\alpha^2\int_\Om \tld{V}_q\Psi_q\dd x+q\alpha\int_\Om \tld{V}_qr\dd x.
	\]
The two integrals of $\tld{V}_qr$ cancel, and the remaining $\frac{q\alpha^2}{2}\int_\Om \tld{V}_q\Psi_q\dd x-q\alpha^2\nl(u_q,u_q)$ equals the third term in \eqref{eq:quad-comparison}, because $2\nl(u_q,u_q)=\int_{B_1}\tld{V}_qu_q\dd x$.

We now estimate the terms in \eqref{eq:quad-comparison}. The moment estimates \eqref{eq:low-modes}--\eqref{eq:radial-moment} used only the $C^2$ bound on $\Psi_q$, its linear vanishing on $\pt B_1$, and the $C^1$ proximity of the ball state to $\psi$, so they remain valid here as well. In the boundary computation \eqref{eq:boundary-comparison}, the $O(q)$ change in the second-order coefficient of $\Psi_q$ produces an additional error, and the error term becomes $O\big((q+\delta)\norm{z}_{H^1(B_1)}^2\big)$. The first term in \eqref{eq:quad-comparison} is $\int_\Om(|\nabla r|^2-\tld{\mu}_qr^2)\dd x$, without the potential $qV_q$ of \eqref{eq:comparison}, and \eqref{eq:bulk-comparison} holds for it by the same change of variables, since $-\tld{\mu}_q\det DT$ differs from $-\pi^2$ by at most $C(q+\delta)$.

The shell term, on the other hand, satisfies
	\[
		\Big|\int_\Om \tld{V}_q\Psi_q\dd x-\int_{B_1}\tld{V}_qu_q\dd x\Big| \leq \norm{\tld{V}_q}_\infty\int_{\pt B_1}\Big|\int_1^{1+\vphi(\theta)}|\Psi_q(\rho)|\rho^2\dd\rho\Big|\dd\Htwo(\theta) \leq C\norm{\vphi}_{L^2}^2,
	\]
because $|\Psi_q(\rho)|\leq C|\rho-1|$. After multiplication by $q\alpha^2/2$, it is bounded by $Cq\norm{z}_{H^1(B_1)}^2$ thanks to \eqref{eq:trace-control}. The last term in \eqref{eq:quad-comparison} is nonnegative. Lemma~\ref{lem:robin} therefore gives the analogue of \eqref{eq:coercive-comparison} for the quadratic energy, and Step 3 of the proof of Theorem~\ref{thm:local} yields \eqref{eq:quad-joint}.

For the upper bound in \eqref{eq:quad-two-sided} we use the normalized transported trial state of Step 4 and \eqref{eq:quad-comparison} with $\alpha=1$. Its $L^2(\R^3)$ distance from $u_q$ is $O(\norm{\vphi}_{H^{1/2}})$ by \eqref{eq:trial-bounds}, and the Hardy--Littlewood--Sobolev inequality on the common bounded support gives 
	\[
		\nl(v-u_q,v-u_q)\leq C\norm{v-u_q}_{6/5}^2\leq C\norm{v-u_q}_2^2.
	\] 
All other terms have the bounds already established, and the lower bound in \eqref{eq:quad-two-sided} follows from \eqref{eq:quad-joint}.

The global assertion follows as in Section~\ref{sec:global}. By \cite[Theorem 1.1]{MMR26H}, for sufficiently small $q$, a minimizer exists and, after a translation, the positivity set $\{u>0\}$ of a nonnegative optimal state $u$ on every minimizer $\Om$ is a radial graph $\Om_\vphi$ with $\norm{\vphi}_{C^{2,\alpha}(\pt B_1)}\to0$ as $q\to0$. The set $\{u>0\}\subset\Om$ has the same energy as $\Om$, by the monotonicity of $\widetilde{\mathcal{E}}_q$ with respect to inclusion, which is immediate from \eqref{eq:quadratic-energy}. In the dilation argument, the Coulomb term scales as $\nl(u_t,u_t)=t^2\nl(u,u)$, so that for $t>1$
	\[
		\widetilde{\mathcal{E}}_q\big(u_t,t\{u>0\}\big) = t^{-2}\int_{\{u>0\}}|\nabla u|^2\dd x + qt^2\nl(u,u),
	\]
and this is smaller than $\widetilde{\mathcal{E}}_q(u,\{u>0\})$ whenever $qt^2\nl(u,u)<\int_{\{u>0\}}|\nabla u|^2\dd x$. Since $\{u>0\}$ is a small radial graph, the factor $t$ which restores the volume $|B_1|$ is close to one, while $\int_{\{u>0\}}|\nabla u|^2\dd x\geq\pi^2$ by the Faber--Krahn inequality and $\nl(u,u)\leq C$ by the Hardy--Littlewood--Sobolev inequality on the bounded support. Hence, for small $q$, the dilation strictly decreases the energy unless $|\{u>0\}|=|B_1|$, and $\{u>0\}$ coincides with $\Om$ up to a null set. The remainder of the proof of Theorem~\ref{thm:global} applies with the positive ball state $u_q$, with the Lipschitz bound $|\nl(v,v)-\nl(u,u)|\leq C\norm{v-u}_2$ on bounded sets of $H^1$ with common bounded support in place of \eqref{eq:D-lip}, and with the simplicity of the first eigenvalue in place of Lemma~\ref{lem:ball}(i).
\end{proof}

\medskip

\begin{remark}[Scaling of the quadratic model]
For $u_R(x)\defeq R^{-3/2}u(x/R)$ one has $\nl(u_R,u_R)=R^2\nl(u,u)$, hence
	\[
		\widetilde{\mathcal{E}}_q(R\Om) = R^{-2}\,\widetilde{\mathcal{E}}_{qR^4}(\Om),
	\]
whereas the quartic scaling is $\en(R\Om)=R^{-2}\mathcal{E}_{qR}(\Om)$. Thus the conclusions of Corollary~\ref{cor:quadratic} at radius $R$ hold under the condition $qR^4\leq q_0$, with the local deficit multiplied by $R^{-2}$. This is the reduction $q=2(m/|B_1|)^{4/3}$ of \cite{MMR26H}.
\end{remark}

\bigskip

\subsection*{Acknowledgments} This project was initiated in 2024 during the author’s stay at LMU Munich as a Humboldt Research Fellow. The problem was suggested by Rupert Frank, and the approach taken here was shaped by many discussions with him; the author is deeply indebted to him for his generosity with his time and ideas. The author also thanks the Mathematical Institute of LMU Munich for its hospitality and the Alexander von Humboldt Foundation for its support. This work was partially supported by the Simons Foundation (MPS-TSM-00851065) and by NSF grants DMS-2306962 and DMS-2607148.

%%%%%%%%%%%%%%%%%%%%%%%%%%%%%%%%%%%%%%%%%%%%%%%%%%%%%%%%%%%%%%%%%%%%%%%%%%%%%%
%%%%%%%%%%%%%%%%%%%%%%%%%%%%%%%%%%%%%%%%%%%%%%%%%%%%%%%%%%%%%%%%%%%%%%%%%%%%%%

\end{document}